\documentclass[english]{article}

\usepackage[english]{babel}
\usepackage{latexsym}
\usepackage{amsthm}
\usepackage{rotating}

\usepackage{multirow}
\usepackage{amsmath}
\usepackage{amsfonts}
\usepackage{amssymb}
\usepackage{verbatim}
\usepackage{psfrag}
\usepackage{graphicx}
\usepackage[all]{xy}

\usepackage{color}

\definecolor{grey}{rgb}{0.6,0.6,0.6}
\newcommand{\proofend}{\hfill $\square$}
\newcommand{\retour}{\ \\ }
\newcommand{\espace}{\ }
\newcommand{\Z}[1]{\mathbb{Z}/#1\mathbb{Z}}

\newcommand{\nuu}[1]{{\small #1}}
\newcommand{\NFC}{\mathrm{NFC}}
\newcommand{\NFCA}{\mathrm{NFCA}}

\newcommand{\DiaG}[4]{[#1:#2:#3:#4]}

\newcommand{\PGL}{\mathrm{PGL}}
\newcommand{\GL}{\mathrm{GL}}
\newcommand{\dJo}{\mathbf{dJo}}
\newcommand{\im}{{\bf i}}
\newcommand{\K}{\mathbb{K}}

\newcommand{\C}{\mathbb{C}}
\newcommand{\Bir}{\mathrm{Bir}}

\newcommand{\Aut}{\mathrm{Aut}}

\newcommand{\gnu}{<\!\nu\!>}

\newcommand{\rkPic}[1]{\mathrm{rk\ Pic}(#1)}
\newcommand{\Pn}{\mathbb{P}^2}
\newcommand{\p}{\mathbb{P}}

\newcommand{\h}[1]{\hspace{-#1mm}}

\newcommand{\numpt}[1]{\ensuremath{#1}}

\makeatletter

\@addtoreset{equation}{section}
\makeatother

\newtheorem{lemm}{Lemma}[section]
\newtheorem{prop}[lemm]{Proposition}

\newtheorem{nota}[lemm]{Notation}
\newtheorem{defi}[lemm]{Definition}
\newtheorem{rema}[lemm]{Remark}
\newtheorem{coro}[lemm]{Corollary}
\newtheorem{theo}{Theorem}

\begin{document}
\title{Elements and cyclic subgroups of finite order of the Cremona group}
\author{J\'er\'emy Blanc}
\maketitle
\begin{abstract}
We give the classification of elements -- respectively cyclic subgroups -- of finite order of the Cremona group, up to conjugation. Natural parametrisations of conjugacy classes, related to fixed curves of positive genus, are provided.
\end{abstract}

\section{Introduction}
The group $\Bir(\Pn)$, called \emph{Cremona group}, is the group of birational transformations of the complex plane $\Pn=\Pn(\C)$. An element (or a subgroup) of $\Bir(\Pn)$ is said to be \emph{of de Jonqui\`eres type} if it preserves a pencil of rational curves; after conjugation the pencil may be chosen to be the set of lines passing through a point.

The study of the elements (or cyclic subgroups) of finite order of the Cremona group is a classical subject. 

\bigskip

The first work in this direction was the classification of involutions given by Bertini \cite{Bertini}. Bertini's proof is incomplete; a short and complete proof was published only recently \cite{BayBea}.

The classification of finite cyclic subgroups of prime order of $\Bir(\Pn)$ was completely achieved a few years ago (see \cite{BayBea}, \cite{Fer} and \cite{BeaBla}, and also  \cite{bib:Zha}). 

The classification of all finite cyclic subgroups which are \emph{not} of de Jonqui\`eres type was almost achieved in \cite{bib:DoI}. A list of representative elements is available: \cite[Table~9]{bib:DoI} (see also \cite[Table~1]{bib:JBCR}); explicit forms are given and the dimension of the varieties which parametrise the conjugacy classes are also provided; there are $29$ families of groups of order at most $30$\footnote[2]{There are in fact $30$ families in \cite[Table 9]{bib:DoI}, but one -- named $D_5$ -- is intrusive since it preserves a rational fibration.}. To complete the classification of these groups there remains, as \cite{bib:DoI} says, to ''Give a finer geometric description of the algebraic variety parametrizing conjugacy classes''.

  In the case of linearisable groups (groups conjugate to subgroups of $\Aut(\Pn)=\PGL(3,\C)$), there is exactly one conjugacy class for each order $n$, represented by the group generated by $(x:y:z)\mapsto (x:y:e^{2\im\pi/n}z)$ (see \cite{BeaBla}). 
  
  The remaining groups -- non-linearisable de Jonqui\`eres groups -- are less well understood so far. It is not too difficult to see that these groups have even order, and that their $2$-torsion is generated by an element which preserves any element of the invariant pencil of rational curves. There are some other geometric descriptions in \cite[Section 5]{bib:DoI}, but nothing more on the conjugacy classes of cyclic groups of de Jonqui\`eres type  (although there are more results on non-cyclic  finite groups of de Jonqui\`eres type). In \cite{bib:JBSMF}, we proved the existence of infinitely many conjugacy classes of such cyclic groups of order $2n$, for any $n$, but the parametrisation of this huge family of groups was not discussed. 
  
Concerning \emph{elements} (instead of cyclic subgroups) of finite order of $\Bir(\Pn)$, there is no further result except that two linearisable elements of the same order are conjugate (\cite{BeaBla}). For example, in the classification of cyclic groups of order $3$ or $5$ given by \cite{bib:Zha} and \cite{Fer}, there is no result concerning the possible conjugation between an element and its inverse.

\bigskip

In the present article, we complete the classification of elements and cyclic subgroups of finite order of $\Bir(\Pn)$. The main contribution concerns de Jonqui\`eres elements and is described in Section~$\ref{Sec:AutConBun}$. We apply cohomology group theory and other simple algebraic tools to the group $\PGL(2,\C(x))\rtimes \PGL(2,\C)$, which is isomorphic to the subgroup of $\Bir(\Pn)$ which preserves some fixed pencil of lines. The classification obtained is summarised in Theorem~$\ref{Thm:ExplicitdJo}$. We also describe the action on conic bundles induced by these elements (\S\ref{SubSec:GeomDesc}), and correct  a wrong assertion  
made in \cite[Theorem 5.7]{bib:DoI}. Section~$\ref{Sec:CyclicGrNotDJ}$ concerns cyclic groups which are not of de Jonqui\`eres type; we use the classification of \cite{bib:DoI} and refine it by providing the parametrisations of the $29$ families of cyclic groups (Theorem~\ref{Thm:ConjAmong29}), using some classical tools on surfaces and curves. Then, we give the conjugacy classes of elements by studying the possible conjugations between the different generators of a cyclic group (Theorem~\ref{Thm:EltGrNotDJ}, proved in Section~\ref{Sec:EltGrNotDJ}).

In such a classification, there are two steps. The first is to find representative families and to prove that each group is conjugate to one of these. The second step is to parametrise the conjugacy classes in each family, by algebraic varieties. For cyclic groups of prime order, the varieties are the moduli spaces of the non-rational curves fixed by the groups (a curve is fixed by a birational map if this map restricts to the identity on the curve). Here, we naturally generalise this invariant, by looking for the non-rational curves fixed by the non-trivial elements of the group. The main invariant, called $\NFCA$, is described in Section~\ref{Sec:results}, before the statements of the principal results. The importance of this invariant appears in Theorem \ref{Thm:NFCA}, proved in Section \ref{Sec:ProofNFCA}.

\smallskip

I thank I.~Dolgachev for his useful remarks on this article.

\section{Results}\label{Sec:results}
\subsection{The conjugacy invariant which determines the conjugacy classes}
The \emph{genus} is for us the geometric genus and we say that a birational map \emph{fixes} a curve if it restricts to the identity on the curve.
\begin{defi}[Normalised fixed curve: $\NFC$ -- already defined in \cite{Fer}]\retour
Let $g\in \Bir(\Pn)$ be a non-trivial element of finite order. If no curve of positive genus is fixed $($pointwise$)$ by $g$, we say that $\NFC(g)=\emptyset$; otherwise $g$ fixes exactly one curve of positive genus {\upshape (\cite{BayBea}, \cite{Fer})}, and $\NFC(g)$ is then the isomorphism class of the normalisation of this curve.
\end{defi}
It was proved in \cite[Proposition 2.7]{BayBea} that two involutions $g_1,g_2\in \Bir(\Pn)$ are conjugate if and only if $\NFC(g_1)=\NFC(g_2)$. If $g_1,g_2\in \Bir(\Pn)$ have the same prime order, then the groups generated by $g_1,g_2$ (but not necessarily $g_1$ and $g_2$, as we will see later in Theorem~\ref{Thm:EltGrNotDJ}) are conjugate if and only if $\NFC(g_1)=\NFC(g_2)$ (see \cite{Fer}, \cite{BeaBla}). For cyclic groups of composite order, the  $\NFC$ invariant is no longer sufficient, as observed in \cite{BeaBla}; we must therefore introduce a new one.

\begin{defi}[Normalised fixed curve with action: $\NFCA$]\retour
Let $g\in \Bir(\Pn)$ be a non-trivial element of finite order $n$. Then, $\NFCA(g)$ is the sequence of isomorphism classes of pairs $$\left((\NFC(g^k),\ g_{|_{\NFC(g^k)}})\right)_{k=1}^{n-1},$$ where $g_{|_{\NFC(g^k)}}$ is the automorphism induced by $g$ on the curve $\NFC(g^k)$ {\upshape (}if\ \ $\NFC(g^k)$ is equal to $\emptyset$, then $g$ acts trivially on it{\upshape )}.
\end{defi}

The following result -- proved in Section~\ref{Sec:ProofNFCA} -- gives a precise and simple way to decide whether two cyclic subgroups (respectively elements) of finite order of $\Bir(\Pn)$ are conjugate.
\begin{theo}[Relation between the conjugacy classes of cyclic groups and the $\NFCA$ invariant]\espace\label{Thm:NFCA}
\begin{enumerate}
\item[\upshape 1.]
Let $G, H\subset \Bir(\Pn)$ be two cyclic subgroups of the same finite order. Then, $G$ and $H$ are conjugate in $\Bir(\Pn)$ if and only if $\NFCA(g)=\NFCA(h)$ for some generators $g$ of $G$ and $h$ of $H$.
\item[\upshape 2.]
Let $g,h\in \Bir(\Pn)$ be two de Jonqui\`eres elements of the same finite order. Then, $g$ and $h$ are conjugate in $\Bir(\Pn)$ if and only if $\NFCA(g)=\NFCA(h)$.
\item[\upshape 3.]
Assertion $2$ is false if neither $g$ nor $h$ is of de Jonqui\`eres type.
\end{enumerate}
\end{theo}
\subsection{The classifications}
We can now state the explicit classification of elements and cyclic groups of finite order of $\Bir(\Pn)$. The following theorem summarises the results of Section~\ref{Sec:AutConBun}.
\begin{theo}[Classification of de Jonqui\`eres elements of finite order]\label{Thm:ExplicitdJo}\espace
\begin{enumerate}
\item[\upshape 1.]
For any positive integer $m$, there exists an unique conjugacy class of linearisable elements of order $m$, represented by the automorphism \[(x:y:z)\mapsto (x:y:e^{2\im\pi/m}z).\]
\item[\upshape 2.]
Any non-linearisable de Jonqui\`eres element of finite order of $\Bir(\Pn)$  has order $2n$, for some positive integer $n$, and is conjugate to an element $g$, such that $g$ and $g^n$ are of the following form:

$$\begin{array}{rcl}g:&(x,y)\dasharrow &\left(e^{2\im\pi/n} x, \frac{a(x)y+(-1)^{\delta}p(x^n)b(x)}{b(x)y+(-1)^{\delta}a(x)}\right),\vspace{0.1 cm} \\ 
g^n:&(x,y)\dasharrow & \left(x,\frac{p(x^n)}{y}\right),\end{array}$$
where $a,b\in \C(x)$, $\delta\in\pm\{0,1\}$, and $p$ is a polynomial with simple roots. The curve $\Gamma$ of equation $y^2=p(x^n)$ -- $($pointwise$)$ fixed by $g^n$ -- is hyperelliptic, of positive genus, and the action of $g$ on this curve has order $n$ and is not a root of the involution associated to any $g_1^2$.

Furthermore, the above association yields a parametrisation of the conjugacy classes of non-linearisable de Jonqui\`eres elements  of order $2n$ of $\Bir(\Pn)$ by isomorphism classes of pairs $(\Gamma,h)$, where $\Gamma$ is a smooth hyperelliptic curve of positive genus, and $h\in \Aut(\Gamma)$ is an automorphism of order $n$, which preserves the fibers of the $g_1^2$ and is not a root of the involution associated to the $g_1^2$.
\end{enumerate}
\end{theo}
\begin{rema}In the above theorem, a hyperelliptic curve is a curve which admits a $(2\! :\! 1)$-map $g_1^2:\Gamma\rightarrow \mathbb{P}^1$. If the curve has genus $\geq 2$, the $g_1^2$ is unique, but otherwise it is not.\end{rema}
\begin{rema}The analogous result for finite de Jonqui\`eres cyclic groups is obvious, and follows directly from Theorem~$\ref{Thm:ExplicitdJo}$.\end{rema}

\bigskip

In the case of elements which are not of de Jonqui\`eres type, the distinction between cyclic groups and elements is more important since in general two generators of the same group are not conjugate, even if the actions on the curves are the same. Before stating the results for the remaining cases (Theorems~\ref{Thm:ConjAmong29} and \ref{Thm:EltGrNotDJ}), we recall the possible normalisations of the non-rational curves fixed by non-de Jonqui\`eres elements of finite order of $\Bir(\Pn)$:

a) elliptic curves: curves of genus $1$ -- {\it abbreviated as "ell" in Table~$\ref{ListN2}$};

b) (hyperelliptic) curves of genus $2$ -- {\it abbreviated as "gen$2$" in Table~$\ref{ListN2}$};

c) Geiser curves: non-hyperelliptic curves of genus $3$ (isomorphic to plane smooth quartic curves) -- {\it abbreviated as "Gei" in Table~$\ref{ListN2}$};

d) Bertini curves: non-hyperelliptic curves of genus $4$ whose canonical model lies on a singular quadric -- {\it abbreviated as "Ber" in Table~$\ref{ListN2}$}.
\begin{rema}
Note that if $\Gamma$ is a curve isomorphic to one of the curves in {\upshape a), b), c)} or {\upshape d)}, there are infinitely many distinct birational embeddings of $\Gamma$ into $\Pn$. The classification below {\upshape (}and also in {\upshape \cite{BayBea}} and {\upshape \cite{Fer}}{\upshape )} implies that there is only one embedding of $\Gamma$ -- up to birational equivalence -- which is fixed $($pointwise$)$ by a birational map of finite order. For example, no smooth quartic curve of $\Pn$ is fixed by an element of $\Bir(\Pn)$, but any smooth quartic curve is birational to a sextic with $7$ double points, which is fixed by a Geiser involution. For more details on this, see {\upshape \cite{bib:BPV2}}.
\end{rema}
The following result will be proved in Section~\ref{Sec:CyclicGrNotDJ}.  
\begin{theo}[Classification of cyclic subgroups of finite order which are not of de Jonqui\`eres type]\espace
\label{Thm:ConjAmong29}
\begin{enumerate}
\item[\upshape 1.]
Any non-trivial cyclic subgroup of finite order of the Cremona group which is not of de Jonqui\`eres type is conjugate to an element of one of the $29$ families of Table~$\ref{ListN1}$. 
\item[\upshape 2.]
If $g$ is a generator of a group of Table~$\ref{ListN1}$ and $h$ is a de Jonqui\`eres element, then $\NFCA(g)\not=\NFCA(h)$.
\item[\upshape 3.]
Any two of the $29$ families of Table~$\ref{ListN1}$ represent distinct families of conjugacy classes of non-de Jonqui\`eres cyclic groups in the Cremona group.
\item[\upshape 4.]
If a curve of positive genus is fixed by some non-trivial element of a group of Table~$\ref{ListN1}$, then its normalisation is either elliptic, or of genus $2$, or is a Bertini or Geiser curve. Each possibility is listed in Table~$\ref{ListN2}$.
\item[\upshape 5.]
In families which contain more than one element {\upshape (}$\#1$ through $\#18$, $\#20$ and $\#23${\upshape )}, the conjugacy classes of  subgroups $G$ belonging to the family are parametrised by isomorphy classes of pairs $(\Gamma, G|_{\Gamma})$, where $\Gamma$ is the normalisation of the unique curve of positive genus fixed by the $r$-torsion of $G$ and $G|_{\Gamma}$ is the action of $G$ on the curve; the number $r$ and the type of the pairs are given in Table~$\ref{ListN2}$. For families $\# 1$ through $\# 5$, $\#7$, $\#8$, $\#12$, $\#13$ and $\#17$, there is only one action  associated to any curve $\Gamma$, so the conjugacy class is parametrised by the curves only.
\item[\upshape 6.]
Let $G$ and $H$ be two cyclic subgroups of the same finite order of the Cremona group, which are not of de Jonqui\`eres type. Then, $G$ and $H$ are conjugate if and only if $\NFCA(g)=\NFCA(h)$ for some generators $g$, $h$ of $G$ and $H$ respectively.

\end{enumerate}
\end{theo}
\begin{sidewaystable}
\centering
\begin{tabular}{|llllll|l|cc}
\hline
\multirow{2}{*}{\begin{sideways}order\end{sideways}} &  {\it not.\ of} & {\it not.\ of} &  {\it generator of} & {\it equation of}& {\it in the }&\#  \\
& {\upshape \cite{bib:JBCR}} &{\upshape \cite{bib:DoI}} & {\it the group}  & {\it the surface} & {\it space }&\\
\hline
2& ${\numpt{2.G}}$ &$A_1^7$& $\DiaG{-1}{1}{1}{1}$& $w^2=L_4(x,y,z)$& $\mathbb{P}(2,1,1,1)$&\nuu{1}  \\
2& ${\numpt{1.B}}$ &$A_1^8$& $\DiaG{-1}{1}{1}{1}$& $w^2=z^3+zL_4(y,z)+L_6(y,z)$& $\mathbb{P}(3,1,1,2)$ &\nuu{2} \\
\hline 
3&${\numpt{3.3}}$ &$3A_2$& $\DiaG{\omega}{1}{1}{1}$& $w^3+L_3(x,y,z)$& $\mathbb{P}^3$&\nuu{3}  \\
3&$\numpt{1.\rho}$ &$4A_2$& $\DiaG{1}{1}{1}{\omega}$& $w^2=z^3+L_6(x,y)$& $\mathbb{P}(3,1,1,2)$&\nuu{4}  \\
\hline 
4&$\numpt{2.4}$ &$2A_3\!+\!A_1$& $\DiaG{1}{1}{1}{\im}$ & $w^2=L_4(x,y)+z^4$ & $\mathbb{P}(2,1,1,1)$&\nuu{5}  \\
4&$\numpt{1.B2.2}$ &$2D_4(a_1)$& $\DiaG{\im}{1}{-1}{-1}$ & $w^2=z^3+zL_2(x^2,y^2)+xyL_2'(x^2,y^2)$ & $\mathbb{P}(3,1,1,2)$&\nuu{6}  \\
\hline 
5&${\numpt{1.5}}$ &$2A_4$& $\DiaG{1}{1}{\zeta_5}{1}$ & $w^2=z^3+\lambda x^4z+x(\mu x^5+y^5)$ & $\mathbb{P}(3,1,1,2)$&\nuu{7}\\
\hline 
6&${\numpt{3.6.1}}$ &$E_6(a_2)$& $\DiaG{\omega}{1}{1}{-1}$& $w^3+x^3+y^3+xz^2+\lambda yz^2$& $\mathbb{P}^3$&\nuu{8}  \\
6&${\numpt{3.6.2}}$ &$A_5\!+\!A_1$& $\DiaG{1}{-1}{\omega}{\omega^2}$& $wx^2+w^3+y^3+z^3+\lambda wyz$& $\mathbb{P}^3$&\nuu{9}  \\
6&${\numpt{2.G3.1}}$ &$E_7(a_4)$& $\DiaG{-1}{1}{1}{\omega}$& $w^2=L_4(x,y)+z^3L_1(x,y)$& $\mathbb{P}(2,1,1,1)$&\nuu{10}  \\
6&${\numpt{2.G3.2}}$ &$D_6(a_2)\!+\!A_1$& $\DiaG{-1}{1}{\omega}{\omega^2}$& $w^2=x(x^3+y^3+z^3)+yzL_1(x^2,yz)$& $\mathbb{P}(2,1,1,1)$&\nuu{11}  \\
6&${\numpt{2.6}}$ &$A_5\!+\!A_2$& $\DiaG{-1}{\omega}{1}{-1}$& $w^2=x^3y+y^4+z^4+\lambda y^2z^2$& $\mathbb{P}(2,1,1,1)$&\nuu{12}  \\
6&$\numpt{1.\sigma\rho}$ &$E_8(a_8)$& $\DiaG{-1}{1}{1}{\omega}$ & $w^2=z^3+L_6(x,y)$ & $\mathbb{P}(3,1,1,2)$ &\nuu{13} \\
6&$\numpt{1.\rho 2}$ &$E_6(a_2)+A_2$& $\DiaG{1}{1}{-1}{\omega}$ & $w^2=z^3+L_3(x^2,y^2)$ & $\mathbb{P}(3,1,1,2)$&\nuu{14}  \\
6&${\numpt{1.B3.1}}$ &$2D_4$& $\DiaG{-1}{1}{\omega}{1}$ & $w^2=z^3+xL_1(x^3,y^3)z+L_2(x^3,y^3)$ & $\mathbb{P}(3,1,1,2)$&\nuu{15}  \\
6&${\numpt{1.B3.2}}$ &$2D_4$& $\DiaG{-1}{1}{\omega}{\omega}$ & $w^2=z^3+\lambda x^2 y^2z+L_2(x^3,y^3)$ & $\mathbb{P}(3,1,1,2)$&\nuu{16}  \\
6&${\numpt{1.6}}$  &$A_5\!+\!A_2\!+\!A_1$& $\DiaG{1}{1}{-\omega}{1}$ & $w^2=z^3+\lambda x^4z+\mu x^6+y^6$ & $\mathbb{P}(3,1,1,2)$&\nuu{17}  \\
\hline 
8&${\numpt{1.B4.2}}$  &$D_8(a_3)$& $\DiaG{\zeta_8}{1}{\im}{-\im}$ & $w^2=z^3+\lambda x^2y^2z +xy(x^4+y^4)$ & $\mathbb{P}(3,1,1,2)$&\nuu{18}  \\ 
\hline 
9&${\numpt{3.9}}$  &$E_6(a_1)$& $\DiaG{\zeta_9}{1}{\omega}{\omega^2}$ & $w^3+xz^2+x^2y+y^2z$ & $\mathbb{P}^3$&\nuu{19}  \\
\hline 
10&${\numpt{1.B5}}$  &$E_8(a_6)$& $\DiaG{-1}{1}{\zeta_5}{1}$ & $w^2=z^3+\lambda x^4z+x(\mu x^5+y^5)$ & $\mathbb{P}(3,1,1,2)$&\nuu{20}  \\
\hline 
12&${\numpt{3.12}}$  &$E_6$& $\DiaG{\omega}{1}{-1}{\im}$ & $w^3+x^3+yz^2+y^2x$ & $\mathbb{P}^3$&\nuu{21}  \\
12&${\numpt{2.12}}$  &$E_7(a_2)$& $\DiaG{1}{\omega}{1}{\im}$ & $w^2=x^3y+y^4+z^4$ & $\mathbb{P}(2,1,1,1)$&\nuu{22}  \\
12&${\numpt{1.\sigma\rho 2.2}}$  &$E_8(a_3)$& $\DiaG{\im}{1}{-1}{-\omega}$ & $w^2=z^3+xyL_2(x^2,y^2)$ & $\mathbb{P}(3,1,1,2)$ &\nuu{23} \\
\hline 
14&${\numpt{2.G7}}$  &$E_7(a_1)$& $\DiaG{-1\h{0.5}}{\h{0.5}\zeta_7\h{0.5}}{\h{0.5}(\zeta_7)^4\h{0.5}}{\h{0.5}(\zeta_7)^2}$ & $w^2=x^3y+y^3z+xz^3$ & $\mathbb{P}(2,1,1,1)$&\nuu{24}\\
\hline 
15&${\numpt{1.\rho 5}}$  &$E_8(a_5)$& $\DiaG{1}{1}{\zeta_5}{\omega}$ & $w^2=z^3+x(x^5+y^5)$ & $\mathbb{P}(3,1,1,2)$ &\nuu{25} \\
\hline 
18&${\numpt{2.G9}}$  &$E_7$& $\DiaG{-1}{(\zeta_9)^6}{1}{\zeta_9}$ & $w^2=x^3y+y^4+xz^3$ & $\mathbb{P}(2,1,1,1)$ &\nuu{26} \\
\hline 
20&${\numpt{1.B10}}$  &$E_8(a_2)$& $\DiaG{\im}{1}{\zeta_{10}}{-1}$ & $w^2=z^3+x^4z+xy^5$ & $\mathbb{P}(3,1,1,2)$ &\nuu{27} \\
\hline 
24&${\numpt{1.\sigma\rho 4}}$  &$E_8(a_1)$& $\DiaG{\zeta_{8}}{1}{\im}{-\im\omega}$ & $w^2=z^3+xy(x^4+y^4)$ & $\mathbb{P}(3,1,1,2)$&\nuu{28}  \\
\hline 
30&${\numpt{1.\sigma\rho 5}}$  &$E_8$& $\DiaG{-1}{1}{\zeta_5}{\omega}$ & $w^2=z^3+x(x^5+y^5)$ & $\mathbb{P}(3,1,1,2)$ &\nuu{29} \\
\hline
\end{tabular}

\caption{The list of finite cyclic subgroups of the Cremona group which are not of de Jonqui\`eres type, viewed as automorphism groups of rational surfaces embedded into weighted projective spaces. The parameters $\lambda, \mu \in \C$ and the homogeneous forms $L_i$ of degree $i$ are such that the surfaces are smooth, and $\DiaG{a}{b}{c}{d}$ denotes the automorphism $(w:x:y:z)\mapsto (aw:bx:cy:dz)$.}\label{ListN1} \end{sidewaystable} 

\begin{sidewaystable}
\centering
\begin{tabular}{|lll|l|l|l|l|l|}
\hline
\multirow{2}{*}{\begin{sideways}order\end{sideways}}  &  {\it generator of} & {\it equation of}& {\it curve }&{\it curves fixed}& \multicolumn{2}{l|}{\it parametrisation by pairs of $r$-torsion}&\#  \\
& {\it the group}  & {\it the surface} & {\it fixed }& {\it by $n$-torsion }& $r$&type of pairs &  \\
\hline
2& $\DiaG{-1}{1}{1}{1}$& $w^2=L_4(x,y,z)$& Geiser&&2&Geiser&\nuu{1}  \\
2& $\DiaG{-1}{1}{1}{1}$& $w^2=z^3+zL_4(y,z)+L_6(y,z)$& Bertini &&$2$&Bertini&\nuu{2} \\
\hline 
3& $\DiaG{\omega}{1}{1}{1}$& $w^3+L_3(x,y,z)$& elliptic &&$3$&elliptic&\nuu{3}  \\
3& $\DiaG{1}{1}{1}{\omega}$& $w^2=z^3+L_6(x,y)$& genus $2$&&$3$&genus $2$&\nuu{4}  \\
\hline 
4& $\DiaG{1}{1}{1}{\im}$ & $w^2=L_4(x,y)+z^4$ & elliptic&&$4$&elliptic &\nuu{5}  \\
4& $\DiaG{\im}{1}{-1}{-1}$ & $w^2=z^3+zL_2(x^2,y^2)+xyL_2'(x^2,y^2)$ && 2:Ber &$2$& (Ber, involution with fixed points)&\nuu{6}  \\
\hline 
5& $\DiaG{1}{1}{\zeta_5}{1}$ & $w^2=z^3+\lambda x^4z+x(\mu x^5+y^5)$ & elliptic&&$5$&elliptic &\nuu{7}\\
\hline 
6& $\DiaG{\omega}{1}{1}{-1}$& $w^3+x^3+y^3+xz^2+\lambda yz^2$&&2,3:ell &$3$&elliptic&\nuu{8}  \\
6& $\DiaG{1}{-1}{\omega}{\omega^2}$& $wx^2+w^3+y^3+z^3+\lambda wyz$& &2:ell&$2$&(ell, gr. of order $3$ without fix.pts.)&\nuu{9}  \\
6& $\DiaG{-1}{1}{1}{\omega}$& $w^2=L_4(x,y)+z^3L_1(x,y)$& &2:Gei, 3:ell&$2$&(Gei, gr. of order 3 with $4$ fix.pts)&\nuu{10}  \\
6& $\DiaG{-1}{1}{\omega}{\omega^2}$& $w^2=x(x^3+y^3+z^3)+yzL_1(x^2,yz)$& &2:Gei&$2$&(Gei, gr. of order $3$ with $2$ fix.pts)&\nuu{11}  \\
6& $\DiaG{-1}{\omega}{1}{-1}$& $w^2=x^3y+y^4+z^4+\lambda y^2z^2$& &3:ell&$3$&elliptic&\nuu{12}  \\
6& $\DiaG{-1}{1}{1}{\omega}$ & $w^2=z^3+L_6(x,y)$ &&2:Ber, 3:gen2 &$3$&gen2&\nuu{13} \\
6& $\DiaG{1}{1}{-1}{\omega}$ & $w^2=z^3+L_3(x^2,y^2)$ &&2:ell, 3:gen2 &$3$ &(gen2, involution with $4$ fix.pts)&\nuu{14}  \\
6& $\DiaG{-1}{1}{\omega}{1}$ & $w^2=z^3+xL_1(x^3,y^3)z+L_2(x^3,y^3)$ &&2:Ber, 3:ell&$2$&(Ber, gr. of order $3$ with $3$ fix.pts) &\nuu{15}  \\
6& $\DiaG{-1}{1}{\omega}{\omega}$ & $w^2=z^3+\lambda x^2 y^2z+L_2(x^3,y^3)$ & &2:Ber&$2$&(Ber, gr. of order $3$ with $1$ fix.pt)&\nuu{16}  \\
6& $\DiaG{1}{1}{-\omega}{1}$ & $w^2=z^3+\lambda x^4z+\mu x^6+y^6$ & elliptic&&$6$&elliptic&\nuu{17}  \\
\hline 
8& $\DiaG{\zeta_8}{1}{\im}{-\im}$ & $w^2=z^3+\lambda x^2y^2z +xy(x^4+y^4)$ & &2:Ber&$2$&(Ber, gr. of order $4$ with $2$ fix.pts)&\nuu{18}  \\ 
\hline 
9& $\DiaG{\zeta_9}{1}{\omega}{\omega^2}$ & $w^3+xz^2+x^2y+y^2z$ & &3:ell&$1$&{\it only one conjugacy class}&\nuu{19}  \\
\hline 
10& $\DiaG{-1}{1}{\zeta_5}{1}$ & $w^2=z^3+\lambda x^4z+x(\mu x^5+y^5)$ & &2:Ber, 5:ell&$2$&(Ber, gr. of order $5$ with $4$ fix.pts)&\nuu{20}  \\
\hline 
12& $\DiaG{\omega}{1}{-1}{\im}$ & $w^3+x^3+yz^2+y^2x$ & &2,3:ell&$1$&{\it only one conjugacy class}&\nuu{21}  \\
12& $\DiaG{1}{\omega}{1}{\im}$ & $w^2=x^3y+y^4+z^4$ & &3,4:ell&$1$&{\it only one conjugacy class}&\nuu{22}  \\
12& $\DiaG{\im}{1}{-1}{-\omega}$ & $w^2=z^3+xyL_2(x^2,y^2)$ && 2:Ber, 3:ell &$2$&(Ber, gr. of order $6$ with $2$ fix.pts)&\nuu{23} \\
\hline 
14& $\DiaG{-1\h{0.5}}{\h{0.5}\zeta_7\h{0.5}}{\h{0.5}(\zeta_7)^4\h{0.5}}{\h{0.5}(\zeta_7)^2}$ & $w^2=x^3y+y^3z+xz^3$ & &2:Ber&$1$&{\it only one conjugacy class}&\nuu{24}\\
\hline 
15& $\DiaG{1}{1}{\zeta_5}{\omega}$ & $w^2=z^3+x(x^5+y^5)$ & &3,5:ell&$1$&{\it only one conjugacy class}&\nuu{25} \\
\hline 
18& $\DiaG{-1}{(\zeta_9)^6}{1}{\zeta_9}$ & $w^2=x^3y+y^4+xz^3$ &&2:Ber&$1$&{\it only one conjugacy class}&\nuu{26} \\
\hline 
20& $\DiaG{\im}{1}{\zeta_{10}}{-1}$ & $w^2=z^3+x^4z+xy^5$ &&2:Ber&$1$&{\it only one conjugacy class} &\nuu{27} \\
\hline 
24& $\DiaG{\zeta_{8}}{1}{\im}{-\im\omega}$ & $w^2=z^3+xy(x^4+y^4)$ &&2:Ber, 3:ell&$1$&{\it only one conjugacy class}&\nuu{28}  \\
\hline 
30& $\DiaG{-1}{1}{\zeta_5}{\omega}$ & $w^2=z^3+x(x^5+y^5)$ &&2:Ber, 5,6:ell&$1$&{\it only one conjugacy class}&\nuu{29} \\
\hline
\end{tabular}

\caption{The parametrisation of the conjugacy classes for each of the $29$ families of Table~$\ref{ListN1}$.}\label{ListN2} \end{sidewaystable} 

Now that the classes of the cyclic groups have been determined, there remains -- in order to give the conjugacy classes of the elements -- to decide whether or not two generators of the same group are conjugate. This is done in the following result, proved in Section~\ref{Sec:EltGrNotDJ}.
\begin{theo}[Classification of elements of finite order which are not of de Jonqui\`eres type]\retour\label{Thm:EltGrNotDJ}
Let  $G$ be a cyclic group acting on the surface $S$, embedded in a weighted projective space {\upshape (}$\mathbb{P}^3$, $\mathbb{P}(2,1,1,1)$ or $\mathbb{P}(3,1,1,2)${\upshape )}, such that the pair $(G,S)$ represents an element of Family $\# n$ of Tables~$\ref{ListN1}$, $\ref{ListN2}$ {\upshape (}with $1\leq n\leq 29${\upshape )}.
Let $g,h$ be two generators of $G$.

 Then, $g,h$ are conjugate in $\Bir(S)$ if and only if they are conjugate in $\Aut(S)$. Furthermore, we have:
\begin{enumerate}
\item[\upshape 1.]
If $n\in\{1,2\}$, then $g=h$.
\item[\upshape 2.]
If $n\in\{9,11,16\}$, then $g$ is conjugate to $h$.
\item[\upshape 3.]
If $n\notin \{6,9,11,16\}$, then $g$ is conjugate to $h$ if and only if $g=h$.
\item[\upshape 4.]
If $n=6$ then the equation of $S$ is given by $w^2=z^3+z(ax^4+bx^2y^2+cx^4)+xy(a'x^4+b'x^2y^2+c'y^4)$ in $\mathbb{P}(3,1,1,2)$ for some $a,b,c,a',b',c'\in \C$. Moreover the two generators of $G$ are conjugate if and only if one of the following occurs:
\begin{center}\begin{tabular}{rl}{\upshape (i)}& $a=c=0$,\\ {\upshape (ii)}& $a'=c'=0$,\\ {\upshape (iii)} & $a,c,a',c'\in\C^{*}$, $a/c=a'/c'$,\\ {\upshape (iv)} & $a,c,a',c'\in\C^{*}$, $a/c=-a'/c'$, and $bb'=0$.\end{tabular}\end{center}
\end{enumerate}
\end{theo}
We conclude this section with a direct consequence of the four theorems.
\begin{coro}\label{Coro:Odd}
Let $g$ be an element of odd order of $\Bir(\Pn)$. Then, $g$ is conjugate to $g^{-1}$ if and only if $g$ is linearisable.\proofend
\end{coro}

\section{The study of de Jonqui\`eres groups}\label{Sec:AutConBun}
\subsection{The group $\dJo$}\label{SubSec:dJo}
The group of birational transformations of the plane that leave invariant a fixed pencil of lines consists, in affine coordinates, of elements of the form
\begin{equation}\Big(x,y\Big)\dasharrow \Big(\frac{ax+b}{cx+d},\frac{\alpha(x)y+\beta(x)}{\gamma(x)y+\delta(x)}\Big),\label{FormdJon}\end{equation}
where $a,b,c,d\in \C$, $\alpha,\beta,\gamma,\delta \in \C(x)$, $ad-bc\not=0$ and $\alpha\delta-\beta\gamma\not=0$. This group is called the \emph{de Jonqui\`eres group}, is denoted by $\dJo$ and is isomorphic to 
$$\PGL(2,\C(x))\rtimes \PGL(2,\C),$$
where the action of $\PGL(2,\C)$ on $\PGL(2,\C(x))$ is the action of $\PGL(2,\C)=\Aut(\mathbb{P}^1(\C))$ on the function field $\C(x)$ of $\mathbb{P}^1(\C)$. The explicit form (\ref{FormdJon}) embeds $\dJo$ into $\Bir(\C^2)$, this latter being conjugate to $\Bir(\Pn)$ after the choice of an embedding $\C^2\hookrightarrow \Pn$.

We will use algebraic tools on the semi-direct product structure of $\dJo$, groups of matrices and Galois cohomology to study its elements of finite order. 
One can proceed in another manner, using geometric tools. 
If $g\in \dJo$ has finite order, the natural embedding of $\C^2$ into $\p^1\times \p^1$ allows us to consider $g$ as an element of $\Bir(\p^1\times\p^1)$ which preserves the first ruling. The blow-up $\eta\colon S\to \p^1\times\p^1$ of all the base-points of the powers of $g$ conjugates $g$ to a biregular automorphism $g'=\eta^{-1} g\eta$ of $S$. Moreover, $g'$ preserves the fibration $\pi\colon S\to \p^1$ which is the composition of $\eta$ with the first projection of $\p^1\times\p^1$. Contracting some curves in fibres of $\pi$, we may assume that $\pi$ is a conic bundle (i.e.\ that every singular fibre of $\pi$ is the union of two transversal $(-1)$-curves), and that $g'$ acts minimally on $(S,\pi)$ (i.e.\ that for any singular fibre $F=F_{1}\cup F_{2}$ of $\pi$, the two components $F_{1}$ and $F_{2}$ are exchanged by some power of $g$). Furthermore, we have $\pi g'=\bar{g}\pi$, where $\bar{g}\in \Aut(\p^1)=\PGL(2,\C)$ corresponds to the image of $g$ by the second projection of $\dJo=\PGL(2,\C(x))\rtimes \PGL(2,\C)$. 

\bigskip

In this section, we characterise the conjugacy classes of the elements of finite order of $\dJo$, and prove Theorem~\ref{Thm:ExplicitdJo}. The first step is to give an explicit form of all non-diagonalisable elements. This is done in Proposition~\ref{Prop:JonqConj2cases}. We show in particular that such an element $g$ has order $2n$, where its action on the basis of the fibration has order $n$, and that $\sigma=g^n$  fixes (pointwise) a curve $\Gamma$, which is a double covering of $\p^1$ by means of the fibration. In particular, $g|_{\Gamma}$ is an automorphism of $\Gamma$ which is not a root of the involution associated to the double covering. This observation is not new and can also be deduced by geometric methods (as in \cite{bib:DoI} and \cite{bib:BlaLin}). In fact, $g$ may be seen as an automorphism of a conic bundle $(S,\pi)$, and $\sigma$ exchanges the two components of $2k$ fibres, where $k\geq 0$, and $\Gamma$ is reducible if $k=0$ and irreducible of genus $k-1$ if $k\geq 1$. Moreover, $\pi$ has $2k$, $2k+1$ or $2k+2$ singular fibres, and all three possibilities occur (see Section~$6$ of \cite{bib:BlaLin}, and in particular Proposition~6.5).

To prove Theorem~\ref{Thm:ExplicitdJo}, we have to go into more detail. We prove that the correspondence between $g$ and the pair  $(\Gamma,g|_{\Gamma})$, endowed with  the action of order $n$ is $1$-to-$1$. Using geometric methods, this is possible when $n$ is equal to $1$ (\cite{BayBea}), but here these methods do not apply directly. We therefore have to use Galois cohomology, as in \cite{bib:BeaP} and \cite{bib:JBSMF}. The correspondence generalises the work in \cite{bib:JBSMF}, which proved the existence of some $g$ of any possible even order $2n$, starting with particular pairs $(\Gamma,g|_{\Gamma})$.

\subsection{Explicit form for non-diagonalisable elements of finite order of $\dJo$}
Although the elements of finite order of $\PGL(2,\C(x))$ and $\PGL(2,\C)$ are quite easy to understand, it is not obvious to describe the conjugacy classes of elements of finite order of $\dJo$. The projection of any subgroup $G\subset\dJo$ on $\PGL(2,\C)$ gives rise to an exact sequence
\[1\rightarrow G_1\rightarrow G \rightarrow G_2\rightarrow 1,\]
where $G_1,G_2$ belong respectively to $\PGL(2,\C(x))$ and $\PGL(2,\C)$. In the case where the exact sequence splits, and $H^1(G_2,\PGL(2,\C(x))=\{1\}$ (which occurs if $G_2$ is finite),
one can conjugate $G$ to $G_1\rtimes G_2\subset \dJo$ \cite[Lemma 2.5]{bib:BeaP}. The hard case is when this exact sequence does not split, even when $G$ is finite and cyclic (which is the case we consider here). The hardest case will be when $G_{1}$ has order $2$, and $G_{2}$ has even order.

 \begin{nota}\label{Not:ActionPGL}
 Note that $\PGL(2,\C(x))$ and $\PGL(2,\C)$ embed naturally into $\dJo$ via respectively the map that associates
 \begin{center}
 \begin{tabular}{llcl}
& $\left(\begin{array}{cc} \alpha & \beta \\  \gamma & \delta \end{array}\right)\in\PGL(2,\C(x))$ & to & $(x,y)\dasharrow \left(x,\frac{\alpha(x)y+\beta(x)}{\gamma(x)y+\delta(x)}\right),$\vspace{0.1cm}\\ 
and & $\left(\begin{array}{cc} a & b \\ c & d \end{array}\right)\in\PGL(2,\C)$&to&$(x,y)\dasharrow \left(\frac{ax+b}{cx+d},y\right).$\end{tabular}
 \end{center}
 
 This induces a semi-direct product structure.
 As in any semi-direct product, any element $\psi\in\dJo$ is equal to $\rho\nu$ for some unique $\rho\in\PGL(2,\C(x))$, $\nu\in \PGL(2,\C)$, and will be written $\psi=\left(\rho,\nu\right)$. This yields the usual multiplication $(\rho,\nu)(\rho',\nu')=(\rho\cdot \nu(\rho'),\nu\nu')$.
 Notice that  $\PGL(2,\C)$ acts on $\PGL(2,\C(x))$ via its action on $\C(x)$ given by $\nu(f)=f\circ \nu^{-1}$. 
 
If $\rho=\left(\begin{array}{cc} \alpha & \beta \\ \gamma & \delta \end{array}\right)\in \PGL(2,\C(x))$ and $\nu=\left(\begin{array}{cc} a & b \\ c & d \end{array}\right)\in \PGL(2,\C)$, the map  $(x,y)\dasharrow \left(\frac{ax+b}{cx+d},\frac{\alpha(x)y+\beta(x)}{\gamma(x)y+\delta(x)}\right)$ is equal to $\nu\rho$ and corresponds to $(\nu\rho\nu^{-1},\nu)=(\nu(\rho),\nu)$; it is equal to $\rho\nu=(\rho,\nu)$ if and only if $\nu(\rho)=\rho$. 
\end{nota}

Firstly, we recall the following basic algebraic result on the group $\PGL(2)$. We leave the proof -- which is an elementary exercise in linear algebra, close to the proof of \cite[Lemma 3.1]{bib:BPV} -- to the reader.
\begin{lemm}\label{Lem:PGL2General}
Let $\K$ be some field of characteristic zero that contains all the roots of unity. 
\begin{enumerate}
\item[\upshape 1.]
Any element $\alpha\in \PGL(2,\K)$ of finite order is either diagonalisable or is an involution conjugate to $\left(\begin{array}{cc} 0 & g \\ 1& 0\end{array}\right)$, for some element $g\in \K^{*}$. \\
Furthermore, the two possibilities are distinct, except for $\left(\begin{array}{cc} 0 & h^2 \\ 1& 0\end{array}\right)$ and $\left(\begin{array}{cc} 1 & 0 \\ 0& -1\end{array}\right)$, which are conjugate, for any $h\in \K^{*}$.
\item[\upshape 2.]
For $a,b\in \K^{*}$, the elements $\left(\begin{array}{cc} 1 & 0 \\ 0& a\end{array}\right),\left(\begin{array}{cc} 1 & 0 \\ 0& b\end{array}\right)\in \PGL(2,\K)$ are conjugate in $\PGL(2,\K)$ if and only if $\vphantom{\Big)}a=b^{\pm 1}$.
\item[\upshape 3.]
For $a,b\in \K^{*}$, the elements $\left(\begin{array}{cc} 0 & a \\ 1& 0\end{array}\right),\left(\begin{array}{cc} 0 & b \\ 1& 0\end{array}\right)\in \PGL(2,\K)$ are conjugate in $\PGL(2,\K)$ if and only if $a/b$ is a square in $\K$.\proofend
\end{enumerate}
\end{lemm}
Secondly, we give a first description of the elements of finite order of $\dJo$, which will be made more precise later.
\begin{prop}\label{Prop:JonqConj2cases}
Let $g\in \dJo=\PGL(2,\C(x))\rtimes \PGL(2,\C)$ be an element of finite order. Then, up to conjugation in the group $\dJo$, $g$ is  one of the following:
\begin{enumerate}
\item[{\upshape 1.}]
A diagonal element $\left(\left(\begin{array}{cc} 1 & 0 \\ 0& \alpha\end{array}\right), \left(\begin{array}{cc} 1 & 0 \\ 0& \beta\end{array}\right)\right)$, for some $\alpha,\beta\in \C^{*}$ of finite order.
\item[{\upshape 2.}]
An element $\left(\rho, \left(\begin{array}{cc} 1 & 0 \\ 0& \xi\end{array}\right)\right)$, where $\xi \in \C^{*}$ is an element of finite order $n\geq 1$, where $g^n$ is equal to the involution $\sigma=\left(\begin{array}{cc} 0 & p(x^n) \\ 1& 0\end{array}\right)\in \PGL(2,\C(x))$, and $p$ is a polynomial with simple roots. Furthermore, $\rho$ commutes  with $\sigma$ in $\PGL(2,\C(x))$ and if $n$ is even we may assume that $p(0)\not=0$.
\end{enumerate}
\end{prop}
\begin{proof}
We write $g=(\rho,\nu)\in \PGL(2,\C(x))\rtimes \PGL(2,\C)$. Observe that the order of $\nu\in \PGL(2,\C)$ is finite, and denote it by $n$; up to a conjugation by an element of  $\PGL(2,\C)$, we may assume that $\nu=\left(\begin{array}{cc} 1 & 0 \\ 0& \xi\end{array}\right)$, where $\xi$ is a primitive $n$-th root of unity.

Assume first that $n=1$, which is equivalent to assuming that $\nu$ is the identity, or that  $g$ belongs to $\PGL(2,\C(x))$. If $g$ is diagonalisable in $\PGL(2,\C(x))$, it is conjugate to $\left(\begin{array}{cc} 1 & 0 \\ 0& \alpha\end{array}\right)$, where $\alpha \in \C(x)^{*}$ is an element of finite order and in particular belongs to $\C^{*}$; this yields an element of type (1). 
If $g$ is not diagonalisable, it is conjugate to $\left(\begin{array}{cc} 0 & p(x) \\ 1& 0\end{array}\right)$, for some $p(x)\in \C(x)^{*}$ which is not a square (Lemma~\ref{Lem:PGL2General}); up to conjugation by a diagonal element of $\PGL(2,\C(x))$ we may assume that $p$ is a polynomial without multiple roots, whence $g$ is of type $(2)$.

Assume now that the order $n$ of $\nu$ is at least equal to $2$.

The element $g^n$ is equal to $(\sigma,1)$, where $\sigma=\rho\cdot \nu(\rho)\cdot\nu^2(\rho)\cdots\nu^{n-1}(\rho)$ is an element of finite order of $\PGL(2,\C(x))$. 
Notice that $\nu(\sigma)=\nu(\rho\cdot \nu(\rho)\cdot\nu^2(\rho)\cdots\nu^{n-1}(\rho))=\rho^{-1}\cdot \sigma \cdot\rho$. Thus, $\nu(\sigma)=\sigma$ if and only if $\rho$ commutes with $\sigma$. This will be crucial in the sequel. We distinguish three cases for $\sigma$, depending on whether its order is $1$, $2$ or at least $3$.

(i) Assume that  $\sigma$ is the identity. This implies that the map $\nu^i\mapsto \rho\cdot \nu(\rho)\cdots \nu^i(\rho)$ represents a cocycle of $\gnu\cong \Z{n}$ with values in $\PGL(2,\C(x))$. The exact sequence $1\rightarrow \C(x)^{*}\rightarrow \GL(2,\C(x))\rightarrow \PGL(2,\C(x))\rightarrow 1$ yields the cohomology exact sequence \[H^1(\gnu,\GL(2,\C(x)))\rightarrow H^1(\gnu,\PGL(2,\C(x)))\rightarrow H^2(\gnu,\C(x)^{*}).\] Since $\C(x)$ has the $\mathrm{C}_1$-property (by Tsen's theorem), $H^2(\gnu,\C(x)^{*})$ is trivial \cite[X.\S 7, Propositions 10 and 11]{bib:Ser}. Furthermore, $H^1(\gnu,\GL(2,\C(x)))=\{1\}$ \cite[X.\S 1, Proposition 3]{bib:Ser}, whence $H^1(\gnu,\PGL(2,\C(x))=\{1\}$. This yields the existence of $\mu \in \PGL(2,\C(x))$ such that 
$\rho=\mu^{-1} \cdot \nu(\mu)$. Consequently, the conjugation of $g$ by $(\mu,1)$ is $(1,\nu)$, which is an element of type $(1)$.

(ii) Assume that $\sigma$ is an element of order $m\geq 3$. Conjugating $g$ by some element of $\PGL(2,\C(x))$, we may suppose that $\sigma$ is  diagonal (Lemma~$\ref{Lem:PGL2General}$), and thus $\sigma=\left(\begin{array}{cc} 1 & 0 \\ 0& c\end{array}\right)$ for some element $c\in \C^{*}$ of finite order.
Since $\nu(\sigma)=\sigma$, $\rho$ commutes with $\sigma$, whence $\rho$ is diagonal, i.e. $\rho=\left(\begin{array}{cc} 1 & 0 \\ 0& r\end{array}\right)$, for some $r\in\C(x)^{*}$.  Choose $b\in \C^{*}$ such that $b^n=c$ and write $r'=r/b$. The equality  $c=r\cdot \nu(r)\cdots \nu^{n-1}(r)$ implies that  $r'\cdot \nu(r')\cdots \nu^{n-1}(r')=c/(b^n)=1$; since $H^1(\Z{n}, \C(x)^{*})=\{1\}$ (by Hilbert's Theorem 90, \cite[X.\S 1]{bib:Ser}), we see that $r'$ is equal to $a^{-1}\cdot \nu(a)$ for some $a\in \C(x)^{*}$, which means that $r=a^{-1}\cdot b\cdot \nu(a)$. The element $\left(\left(\begin{array}{cc} 1 & 0 \\ 0& a\end{array}\right),1\right)$  conjugates $g$ to $\left(\left(\begin{array}{cc} 1 & 0 \\ 0& b\end{array}\right), \left(\begin{array}{cc} 1 & 0 \\ 0& \xi\end{array}\right)\right)$, which is an element of type $(1)$. 

(iii) The remaining case is when $\sigma$ is an involution. After a conjugation by an element of $\PGL(2,\C(x))$, we may assume that $\sigma=\left(\begin{array}{cc} 0 & h \\ 1& 0\end{array}\right)$, for some element $h\in \C(x)^{*}$ (Lemma \ref{Lem:PGL2General}).
Choosing an element $R\in \GL(2,\C(x))$ whose class in $\PGL(2,\C(x))$ is $\rho$, we have $\det(R) \cdot \nu(\det R)\cdots \nu^{n-1}(\det R)=\mu^2 h$, for some $\mu \in \C(x)^{*}$, so $\mu^2h\in \C(x)^{\nu}$. Replacing $\sigma$ by $\left(\begin{array}{cc} 0 & \mu^2h \\ 1& 0\end{array}\right)$, after a conjugation by an element of $\PGL(2,\C(x))$ (Lemma \ref{Lem:PGL2General}), we may assume that $\nu(h)=h$. This condition is equivalent to the fact that $h\in \C(x^n)$. Since the conjugation of $\sigma$ by $\left(\begin{array}{cc} 1 & 0 \\ 0& q(x^n)\end{array}\right)$ changes $h$ into $h\cdot(q(x^n))^2$, for any $q\in\C(x)$, we may assume that $h=p(x^n)$ where $p$ is a polynomial with simple roots. Finally, if $n$ is even and $p(0)=0$, then $x^n$ divides $p(x^n)$. The conjugation of  $\sigma$ by $\left(\begin{array}{cc} 1 & 0 \\ 0& x^{-n/2}\end{array}\right)$ removes the factor $x^n$. We obtain in any case an element of type $(2)$.
\end{proof}
\begin{coro}\label{Cor:DescrRoots}
Up to conjugation in $\dJo$, any element $g$ which is not diagonalisable has order $2n$, for some positive integer $n$, and $\sigma=g^n$ is equal to $(x,y)\dasharrow (x,p(x^n)/y)$, for some polynomial $p$ with simple roots. The curve $\Gamma\subset \C^2$ fixed by $\sigma$ has equation $y^2=p(x^n)$, and is a double covering of $\C$ by means of the $x$-projection $(x,y)\mapsto x$. Moreover, $g$ restricts to an automorphism $g|_{\Gamma}$ of $\Gamma$  of order $n$, which is not a root of the involution associated to the $g_1^2$.\proofend
\end{coro}
Lemma \ref{Lem:PGL2General} and Proposition \ref{Prop:JonqConj2cases} show the importance of involutions in the de Jonqui\`eres group, and especially of those that belong to $\PGL(2,\C(x))$. According to Corollary~\ref{Cor:DescrRoots}, we can associate to any non-diagonalisable element $g$ of $\dJo$ an involution $\sigma\in\PGL(2,\C(x))$, a curve $\Gamma$, and an element $g|_{\Gamma}\in\Aut(\Gamma)$ which is not a root of the involution associated to the double covering $\Gamma\to \C$.

To prove Theorem~\ref{Thm:ExplicitdJo}, we will show that the  map $g\to(\Gamma,g|_{\Gamma})$ is a  $1$-to-$1$ correspondence, i.e. that for any pair $(\Gamma,g|_{\Gamma})$, there exists $g\in \dJo$, and that its conjugacy class in $\dJo$ is unique. This is done in the next subsection, and in particular in Propositions~\ref{Prop:ConjNsigma} and~\ref{Prp:ConjIFFsameAction}. 

\subsection{The map $g\to (\Gamma,g|_{\Gamma})$ is a $1$-to-$1$ correspondence}

The correspondence $\sigma\leftrightarrow \Gamma$ is well-known to be $1$-to-$1$. According to Proposition~\ref{Prop:JonqConj2cases}, $g$ may be viewed as $\left(\rho, \left(\begin{array}{cc} 1 & 0 \\ 0& \xi\end{array}\right)\right)\in\PGL(2,\C(x))\rtimes\PGL(2,\C)$, where $\rho$ commutes with $\sigma$. To prove Propositions~\ref{Prop:ConjNsigma} and~\ref{Prp:ConjIFFsameAction}, we will need the description of the centraliser $N_{\sigma}$ of $\sigma$ in $\PGL(2,\C(x))$ (Lemma~$\ref{Lem:Centraliser}$), and a technical lemma on exact sequences associated to $N_{\sigma}$ (Lemma~\ref{Lem:TechniqExactS}).
\begin{lemm}\label{Lem:Centraliser}
Let $\sigma \in \PGL(2,\C(x))$ be the involution $\left(\begin{array}{cc} 0& h \\ 1& 0\end{array}\right)$ for $h\in \C(x)^{*}$ and denote by $N_\sigma$ the centraliser of $\sigma$ in $\PGL(2,\C(x))$.

\begin{enumerate}
\item[\upshape 1.]
$N_{\sigma}$ is the disjoint union of 
$N_{\sigma}^{0}$ and $N_{\sigma}^{1}$ $($written $N_{\sigma}=N_{\sigma}^{0}\uplus N_{\sigma}^{1})$, where \[N_{\sigma}^{\delta}=\left\{\left(\begin{array}{cc} a& (-1)^\delta bh \\ b& (-1)^\delta a\end{array}\right)\in \PGL(2,\C(x)), a,b\in \C(x)^{*}\right\};\]
\item[\upshape 2.]
$N_{\sigma}=N_{\sigma}^0\rtimes \Z{2}$, where $\Z{2}$ is generated by $\left(\begin{array}{cc} 1& 0\\ 0& -1\end{array}\right)$ and acts on the group $N_{\sigma}^0$ by mapping an element on its inverse;
\item[\upshape 3.]
if $h=\mu^2$, $\mu\in\C(x)^{*}$, the conjugation by $\left(\begin{array}{lr} 1& \mu \\ 1& -\mu\end{array}\right)$ sends $\sigma$, $N_{\sigma}^0$ and $N_{\sigma}^1$ respectively on \[\left(\begin{array}{lr} 1& 0 \\ 0& -1\end{array}\right)\mbox{, }\left\{\left(\begin{array}{cc} 1& 0 \\ 0& a\end{array}\right), a\in \C(x)^{*}\right\}\mbox{ and }\left\{\left(\begin{array}{cc} 0& a \\ 1& 0\end{array}\right), a\in \C(x)^{*}\right\};\]
\item[\upshape 4.]
if $h$ is not a square in $\C(x)^{*}$, the map $\left(\begin{array}{lr} a& bh \\ b& a\end{array}\right)\rightarrow [a+b\sqrt{h}]$ yields an isomorphism from $N_{\sigma}^0$ to $\C(x)[\sqrt{h}]^{*}/\C(x)^{*}$;
\item[\upshape 5.]
each element of $N_{\sigma}$ acts on the curve $\Gamma\subset \C^2$ of equation $y^2=h(x)$. This action gives rise to a split exact sequence
\[1\rightarrow N_{\sigma}^0\rightarrow N_{\sigma}\rightarrow \Z{2}\rightarrow 1.\]
\end{enumerate}
\end{lemm}
\begin{proof}
Assertions $1$ through $4$ follow from a straightforward calculation. Let us prove the last one. Let $p=(x,y)\in \C^2$ be a general point of $\Gamma$. Then, $\left(\begin{array}{lr} a& bh \\ b& a\end{array}\right)$ sends $p$ onto $\left(x,\frac{a(x)y+b(x)h(x)}{b(x)y+a(x)}\right)=\left(x,\frac{y(a(x)y+b(x)h(x))}{b(x)y^2+a(x)y}\right)$, which is equal to $p$ since $y^2=h(x)$. Furthermore, $\left(\begin{array}{lr} 1& 0 \\ 0& -1\end{array}\right)$ sends $p$ onto $(x,-y)\in \Gamma\backslash p$.
\end{proof}
Given some integer $n$, some primitive $n$-th root of unity $\xi$, and some $\sigma:(x,y)\dasharrow (x,y/p(x^n))$ as before, we look for elements $g=\left(\rho, \nu\right)\in\PGL(2,\C(x))\rtimes\PGL(2,\C)$ such that $\rho\in N_{\sigma}$, $\nu=\left(\begin{array}{cc} 1 & 0 \\ 0& \xi\end{array}\right)$ and  $g^n=\sigma$.  This means that $\sigma$ is equal to $\rho\cdot \nu(\rho)\cdots\nu^{n-1}(\rho)$, which corresponds to some kind of norm of $\rho$, according to the action of $\nu$ on $\PGL(2,\C(x))$. To prove the existence of $\rho$, we have to prove that this norm is surjective. If $(\rho,\nu)^n=(\rho',\nu)^n=\sigma$ and $\rho,\rho'$ have the same action on $\Gamma$ (which corresponds to saying that $\rho,\rho'\in N_{\sigma}^i$ for the same $i$, by Lemma~\ref{Lem:Centraliser}), we want to prove that $(\rho,\nu)$ and $(\rho',\nu)$ are conjugate in $\dJo$. We will use the fact that $\rho\cdot (\rho')^{-1}$ has norm $1$. The two results needed correspond in fact to exactness of exact sequences associated to norms, which is proved in Lemma~\ref{Lem:TechniqExactS} below.

 In this lemma, $G^{\nu}$ denotes the set of elements of $G$ fixed by the action of $\nu$.

\begin{lemm}\label{Lem:TechniqExactS}
Let $\xi\in\C^{*}$ be an $n$-th root of unity, let $p\in\C[x]$ be a polynomial with simple roots, let $\sigma=\left(\begin{array}{cc} 0 & p(x^n) \\ 1& 0\end{array}\right)\in\PGL(2,\C(x))$ and let $\nu=\left(\begin{array}{cc} 1 & 0 \\ 0& \xi\end{array}\right)\in\PGL(2,\C)$. Denote by $N_{\sigma}=N_{\sigma}^0\uplus N_{\sigma}^{1}$ the centraliser of $\sigma$ in $\PGL(2,\C(x))$, as in Lemma~$\ref{Lem:Centraliser}$. 

Then, the action of $\nu$ on $\PGL(2,\C(x))$ leaves $N_{\sigma}^0$ invariant and, writing $G=N_{\sigma}^0$,  the following sequence of group homomorphisms is exact:
\begin{eqnarray}
&\xymatrix{G\ar@{->}^{\mbox{\large $d_1$}}_{\alpha \mapsto \alpha\cdot \nu(\alpha)^{-1}}[rrr]&&&G\ar@{->}^{\mbox{\large $\varphi_1$}}_{\alpha \mapsto \alpha\cdot \nu(\alpha)\cdots \nu^{n-1}(\alpha)}[rrrr]&&&&(G)^{\nu}\ar@{->}[r]&1.}\label{eqHstd}
\end{eqnarray}
Furthermore, if $n$ is even each of the following sequences is also exact:
\begin{eqnarray}
&\xymatrix{G\ar@{->}^{\mbox{\large $d_2$}}_{\alpha \mapsto \alpha\cdot \nu^2(\alpha)^{-1}}[rr]&&G\ar@{->}^{\mbox{\large $\varphi_2$}}_{\alpha \mapsto \alpha\cdot \nu^2(\alpha)\cdots \nu^{n-4}(\alpha)\cdot\nu^{n-2}(\alpha)}[rrrr]&&&&(G)^{\nu^2}\ar@{->}[r]&1;}\label{eqH2}\\
&\xymatrix{(G)^{\nu^2}\ar@{->}^{\mbox{\large $d_3$}}_{\alpha \mapsto \alpha\cdot \nu(\alpha)^{-1}}[rr]&&(G)^{\nu^2}\ar@{->}^{\mbox{\large $\varphi_3$}}_{\alpha \mapsto \alpha\cdot \nu(\alpha)\cdots \nu^{n-1}(\alpha)}[rrrr]&&&&(G)^{\nu}\ar@{->}[r]&1;}\label{eqHstd2}\\
&\xymatrix{G\ar@{->}^{\mbox{\large $r$}}_{\alpha \mapsto \alpha\cdot \nu(\alpha)}[rr]&&G\ar@{->}^{\mbox{\large $\psi$}}_{\alpha \mapsto \alpha\cdot \frac{1}{\nu(\alpha)}\cdot \nu^2(\alpha)\cdots \frac{1}{\nu^{(n-1)}(\alpha)}}[rrrr]&&&&(G)^{\nu^2}\ar@{->}^{\mbox{\large $\varphi_3$}}[r]&G^{\nu}\ar@{->}[r]&1.\label{eqHinv}}
\end{eqnarray}
\end{lemm}
\begin{proof}
Using Lemma~$\ref{Lem:Centraliser}$, we make the following observations. If $p(x^n)$ is a square in $\C(x)^{*}$, then $G=N_{\sigma}^0\cong \C(x)^{*}$ and the action of $\nu$ on $G$ is the restriction of the field homomorphism $\C(x)\rightarrow \C(x)$ that sends $f(x)$ on $f(\xi x)$ (see Notation~\ref{Not:ActionPGL}). Otherwise, writing $h=p(x^n)$ we have $G\cong \C(x)[\sqrt{h}]^{*}/\C(x)^{*}$, where the action of $\nu$ is induced by the field isomorphism of $\C(x)[\sqrt{h}]$ that sends $a(x)+b(x)\sqrt{h}$ on $a(\xi x)+b(\xi x)\sqrt{h}$.

Since $\C(x)$ has the $\mathrm{C}_1$-property (by Tsen's theorem), for any finite Galois extension $L/K$, where $K$ is finite over $\C(x)$, the norm of the extension is surjective (see \cite{bib:Ser}, X.\S 7, Propositions 10 and 11). This implies the surjectivity of the homomorphisms $\varphi_1,\varphi_2$ and $\varphi_3$. 

The exactness of sequences \ref{eqHstd}, \ref{eqH2}, \ref{eqHstd2} is respectively equivalent to the triviality of  $H^1(\gnu,G)$, $H^1(<\nu^2>,G)$ and $H^1(\gnu, (G)^{\nu^2})$. If $G\cong \C(x)^{*}$, this follows directly from Hilbert's Theorem $90$. If $G\cong (\C(x)[\sqrt{h}])^{*}/\C(x)^{*}$, the  exact sequence 
\begin{equation}1\rightarrow\C(x)^{*}\rightarrow (\C(x)[\sqrt{h}])^{*}\rightarrow (\C(x)[\sqrt{h}])^{*}/\C(x)^{*}\rightarrow 1\label{ExactSequenceCxsqrth}\end{equation}
yields the cohomology exact sequence
\begin{equation}H^1(\gnu,(\C(x)[\sqrt{h}])^{*})\rightarrow H^1(\gnu,(\C(x)[\sqrt{h}])^{*}/\C(x)^{*})\rightarrow H^2(\gnu,\C(x)^{*}).\label{CohExactSequenceCxsqrth}\end{equation}

Once again, $H^1(\gnu,(\C(x)[\sqrt{h}])^{*})$ is trivial by Hilbert's Theorem 90, and so is $H^2(\gnu,\C(x)^{*})$ by Tsen's theorem  and \cite[X.\S 7, Propositions 10 and 11]{bib:Ser}; we obtain therefore the triviality of $H^1(\gnu,(\C(x)[\sqrt{h}])^{*}/\C(x)^{*})$ and the exactness  of $(\ref{eqHstd})$. The cases $(\ref{eqH2})$ and $(\ref{eqHstd2})$ are similar.

Let us now prove the exactness of $(\ref{eqHinv})$.  The equality $\psi=d_3\circ \varphi_2$ and the exactness of $(\ref{eqH2})$ and $(\ref{eqHstd2})$ imply that the image of $\psi$ is the kernel of $\varphi_3$ and that $\varphi_3$ is surjective. To prove the exactness of $(\ref{eqHinv})$, we separate two possibilities. 

Firstly, assume that $h=p(x^n)$ is not a square, which means that $G\cong \C(x)[\sqrt{h}]^{*}/\C(x)^{*}$ and that $\nu$ acts as the field automorphism of $\C(x)[\sqrt{h}]$ that sends $a(x)+b(x)\sqrt{h}$ on $a(\xi x)+b(\xi x)\sqrt{h}$. Denote by $\overline{\nu}:\C(x)[\sqrt{h}]\rightarrow \C(x)[\sqrt{h}]$ the field automorphism that sends $a(x)+b(x)\sqrt{h}$ on $a(\xi x)-b(\xi x)\sqrt{h}$. Since the order of $\nu$ is even, the map $\nu^i\rightarrow \overline{\nu}^i$ induces another action of $\gnu\cong\Z{n}$ on the field $\C(x)[\sqrt{h}]$. Furthermore, the actions of $\nu$ on $\C(x)^{*}$ and on $(\C(x)[\sqrt{h}])^{*}/\C(x)^{*}$ given respectively by $f(x)\mapsto f(\xi x)$ and $[a(x)+b(x)\sqrt{h}]\mapsto [a(\xi x)-b(\xi x)\sqrt{h}]$ are compatible with the exact sequence (\ref{ExactSequenceCxsqrth}), and thus yield the exact sequence (\ref{CohExactSequenceCxsqrth}) with this new action. Once again, the cohomological sets  $H^1(\gnu,(\C(x)[\sqrt{h}])^{*})$ and $H^2(\gnu,\C(x)^{*})$ associated to these actions are trivial, which implies that $H^1(\gnu,(\C(x)[\sqrt{h}])^{*}/\C(x)^{*})$ is trivial. Observe now that $(a+b\sqrt{h})(a-b\sqrt{h})\in \C(x)$ for any $a,b \in \C(x)$, which implies that the new action of $\nu$ on $(\C(x)[\sqrt{h}])^{*}/\C(x)^{*}$ is $\alpha\mapsto \nu(\alpha)^{-1}$. The triviality of the cohomology group is therefore equivalent to the left exactness of (\ref{eqHinv}).

Secondly, assume that $p(x^n)$ is a square. Since $\psi\circ r$ is the trivial map, there remains to show that if $\beta\in G$ and $\psi(\beta)=1$, then $\beta$ is in the image of $r$. By the equality $\psi=\varphi_2\circ d_1$, and the exactness of (\ref{eqH2}), there exists $\alpha\in G$ such that $d_1(\beta)=d_2(\alpha)$. Since $d_1(r(\alpha))=d_2(\alpha)$, we may replace $\beta$ by  $\beta/ r(\alpha)$ and assume that $\beta$ is in the kernel of $d_1$, i.e. that $\beta\in (G)^{\nu}$. The hypothesis on $p(x^n)$ implies that $G\cong \C(x)^{*}$ and that $G^{\nu}=\C(x^n)^{*}$. Therefore, $\beta=b_0\cdot \prod_{i=1}^r (b_i-x^n)$ where $b_i\in \C$ for $i=0,\dots,r$ and $b_0\not=0$. Then, $\beta$ is equal to $r(a_0\cdot \prod_{i=1}^r (a_i-x^m))$, where $(a_i)^2=b_i$ and $2m=n$ (recall that $n$ is assumed to be even).
\end{proof}
\begin{prop}\label{Prop:ConjNsigma}
Let $\xi\in\C^{*}$ be an $n$-th root of unity, let $p\in\C[x]$ be a polynomial with simple roots, let $\sigma=\left(\begin{array}{cc} 0 & p(x^n) \\ 1& 0\end{array}\right)\in\PGL(2,\C(x))$ and let $\nu=\left(\begin{array}{cc} 1 & 0 \\ 0& \xi\end{array}\right)\in\PGL(2,\C)$. Denote by $N_{\sigma}=N_{\sigma}^0\uplus N_{\sigma}^{1}$ the centraliser of $\sigma$ in $\PGL(2,\C(x))$, as in Lemma~$\ref{Lem:Centraliser}$. Then, the following hold:

\begin{enumerate}
\item[\upshape 1.]
There exists an element $\rho_0 \in N_{\sigma}^0$ such that $(\rho_0, \nu)^n=(\sigma,1)$.

Moreover, if two such elements $\rho_0$ and $\rho_0'$ exist, then $(\rho_0,\nu)$ and $(\rho_0',\nu)$ are conjugate in $\dJo=\PGL(2,\C(x))\rtimes \PGL(2,\C)$.
\item[\upshape 2.]
There exists an element $\rho_1\in N_{\sigma}^1$ such that $(\rho_1,\nu)^n=(\sigma,1)$ if and only if $n$ is even.

Furthermore, if two such elements $\rho_1$ and $\rho_1'$ exist, then $(\rho_1,\nu)$ and $(\rho_1',\nu)$ are conjugate in $\dJo$.
\end{enumerate}
\end{prop}
\begin{proof}
Note that the action of $\nu$ on $\PGL(2,\C(x))$ leaves $N_{\sigma}^{i}$ invariant for $i=0,1$; we denote as before by $(N_{\sigma}^{i})^{\nu}$ the fixed part of each set, and have $\sigma\in (N_{\sigma}^0)^{\nu}$. 

For any element $\rho \in N_{\sigma}^0$, we have $(\rho,\nu)^n=(\rho\cdot \nu(\rho)\cdots \nu^{n-1}(\rho),1)$. The exactness of the sequence  (\ref{eqHstd}) implies the existence of $\rho_0\in N_{\sigma}^0$ such that $(\rho_0,\nu)^n=(\sigma,1)$. Assume now that $\rho_0,\rho_0'\in N_{\sigma}^0$ are such that $(\rho_0,\nu)^n=(\rho_0',\nu)^n=(\sigma,1)$. Writing $\mu=(\rho_0)^{-1}\rho_0'$, we have $\mu\cdot \nu(\mu)\cdots\nu^{n-1}(\mu)=1$, and the exactness of (\ref{eqHstd}) yields the existence of $\lambda \in N_{\sigma}^0$ such that $\mu=\lambda\nu(\lambda)^{-1}$. Then, the elements $(\rho_0,\nu)$  and  $(\rho_0',\nu)$  are conjugate by $(\lambda,1)$. This yields Assertion $(1)$.

If $n$ is odd and $\rho_1\in N_{\sigma}^1$, then $(\rho_1,\sigma)^n$ belongs to $N_{\sigma}^1$ and thus is not equal to~$\sigma$. Assume now that $n$ is even and write $\delta=\left(\begin{array}{cc} 1 & 0 \\ 0& -1\end{array}\right)\in \PGL(2,\C(x))$.
For any $\rho \in N_{\sigma}^0$, we have $\rho\delta \in N_{\sigma}^1$ and  $(\rho\delta,\nu)^n=(\rho\cdot \nu(\rho)^{-1}\cdot \nu^2(\rho)\cdots \nu^{n-1}(\rho)^{-1},1)$. The exactness of the sequence  (\ref{eqHinv}) implies the existence of $\rho_1\in N_{\sigma}^1$ such that $(\rho_1,\nu)^n=(\sigma,1)$. Assume now that $\rho_1,\rho_1'\in N_{\sigma}^1$ are such that $(\rho_1,\nu)^n=(\rho_1',\nu)^n=(\sigma,1)$. Write $\rho_1=\rho_0\delta$ and $\rho_1'=\rho_0'\delta$, for some $\rho_0,\rho_0'\in N_{\sigma}^0$ and set $\mu=(\rho_0)^{-1}(\rho_0)'\in N_{\sigma}^0$. We have $\mu\cdot \nu(\mu)^{-1}\cdot \nu^2(\mu)\cdots\nu^{n-1}(\mu)^{-1}=1$, and the exactness of (\ref{eqHinv}) yields the existence of $\lambda \in N_{\sigma}^0$ such that $\mu=\lambda\nu(\lambda)$. Then, the elements $(\rho_0\delta,\nu)$  and  $(\rho_0'\delta,\nu)$  are conjugate by $(\lambda,1)$. This yields Assertion $(2)$.
\end{proof}
We can now consider the result of Proposition \ref{Prop:ConjNsigma} from the geometric point of view, and obtain the key result of this section.
\begin{prop}\label{Prp:ConjIFFsameAction}
For $i=1,2$, let $h_i\in\dJo$ be a non-diagonalisable element of finite order $2n$ of the de Jonqui\`eres group, whose action on the basis of the fibration has order $n$ and is such that $(h_i)^n$ is equal to an involution of $\PGL(2,\C(x))$ which fixes a $($possibly reducible$)$ curve $\Gamma_i\subset \C^2$ such that the projection on the first factor induces a $(2\! :\! 1)$-map $\pi_i:\Gamma_i\rightarrow \C$. 
Then, the following conditions are equivalent:
\begin{enumerate}
\item[\upshape 1.]
The elements $h_1$ and $h_2$ are conjugate in the de Jonqui\`eres group $\dJo$.
\item[\upshape 2.]
There exist two birational maps $\psi:\Gamma_1\dasharrow \Gamma_2$, $\alpha:\C\dasharrow \C$ such that $\pi_2\psi=\alpha\pi_1$ and $\psi$ conjugates the restriction of $h_1$ on $\Gamma_1$ to the restriction of $h_2$ on $\Gamma_2$.
\end{enumerate}
\end{prop}
\begin{proof}
Assume that $\Psi \in \dJo$ conjugates  $h_1$ to $h_2$. Then, $\Psi$ conjugates $(h_1)^n$ to $(h_2)^n$ and consequently sends $\Gamma_1$ on $\Gamma_2$ birationally. Thus $(1)\Rightarrow(2)$.

Assume the existence of $\psi$ and $\alpha$ as in Assertion 2. Up to conjugation in $\dJo$ we may assume (Proposition \ref{Prop:JonqConj2cases}) that $h_1=(\rho_1,\nu)$ and $(h_1)^n=\sigma$,  where $\nu=\left(\begin{array}{cc} 1 & 0 \\ 0& \xi\end{array}\right)$, $\sigma=\left(\begin{array}{cc} 0 & p(x^n) \\ 1& 0\end{array}\right)$, $\xi\in\C^{*}$ is a primitive $n$-th root of unity and $p$ is a polynomial with simple roots.

There exists an element $\Psi \in \dJo$ that extends $\psi$ and sends $\Gamma_{1}$ on $\Gamma_{2}$, hence conjugates $(h_1)^n$ to $(h_2)^n$ (this follows from the correspondence between the involutions and the hyperrelliptic curves, and can be found by hand directly, using Lemma~\ref{Lem:PGL2General}). 
We may therefore assume, after conjugation of $h_2$ by $\Psi$, that $\Gamma_2=\Gamma_1$,  $(h_2)^n=(h_1)^n=\sigma$, $h_2=(\rho_2,\nu)$ and that $h_1,h_2$ have the same action on $\Gamma_1$. Since $\nu(\sigma)=\sigma$ and $\sigma=\rho_i\cdot \nu(\rho_i)\cdots\nu^{n-1}(\rho_i)$ for $i=1,2$, both $\rho_1$ and $\rho_2$ belong to the centraliser $N_{\sigma}$ of $\sigma$ in $\PGL(2,\C(x))$. The fact that $\rho_1$ and $\rho_2$ have the same action  on $\Gamma_1$ implies that they belong either both to $N_{\sigma}^0$ or both to $N_{\sigma}^1$ (Lemma~$\ref{Lem:Centraliser}$(5)). Finally, we apply Proposition \ref{Prop:ConjNsigma} to deduce that $h_1$ and $h_2$ are conjugate in the de Jonqui\`eres group.
\end{proof}
\subsection{Geometric description -- action on conic bundles}\label{SubSec:GeomDesc}
Let $g\in \dJo$ be an element which is not diagonalisable in this group. Then, $g$ has even order $2n$ and $g^n$ fixes a bisection $\Gamma$ of the fibration.
In \S\ref{SubSec:dJo}, we explained how to see $g$ acting on a conic bundle. We prove now that the action of $(\Gamma,g|\Gamma)$ not only gives the conjugacy class in $\dJo$ (Proposition~\ref{Prp:ConjIFFsameAction}), but also describes the action of $g$ on the conic bundle. Recall that a singular fibre is twisted by an automorphism if this one exchanges the two components of the singular fibre. Assertions $1$ and $3$ of the following Lemma are well-known, whereas the two others are new.
\begin{lemm}\label{Lem:GeomCB}
Let $S$ be a smooth projective surface $S$, endowed with a conic bundle structure $\pi\colon S\to \p^1$. Let $g\in \Aut(S)$ be an automorphism of $S$ of finite order which preserves the set of fibres of $\pi$. Assume that $g$ has order $2n$, where $n$ is the order of the action of $g$ on the basis of the fibration and that the triple $(g,S,\pi)$ is minimal $($i.e.\ that there exists no set of disjoint $(-1)$-curves invariant by $g$ and contained in a finite set of fibres$)$.

Denote by $\Gamma\subset S$ the bisection fixed $($pointwise$)$ by $h=g^n$; let $p\in \mathbb{P}^1$, and let $F=\pi^{-1}(p)\subset S$. Then, the following occur:

\begin{enumerate}
\item[$1$.]
$F$ is a singular fibre twisted by $h$ if and only if $F\cap \Gamma$ consists of one point. 
\item[$2.$]
$F$ is a singular fibre twisted by $g$ and not by $h$ if and only if $F\cap \Gamma$ consists of two points, exchanged by $g|_\Gamma$.
\item[$3.$]
The number of singular fibres twisted by $h$ is equal to $2k$, where $k=0$ if $\Gamma$ is reducible and $k=g(\Gamma)-1$ when $\Gamma$ is irreducible, where $g(\Gamma)$ is the geometric genus of $\Gamma$.
\item[$4$.]
The number of singular fibres of $\pi$ is equal to $2k+r$, where $r\in\{0,1,2\}$. The number $r$ is equal to the number of singular fibres twisted by $g$ and not by $h$. 
\end{enumerate}
\end{lemm}
\begin{proof}
If $F$ is singular and twisted by $h$, then $h$ fixes one point of $F$, so $\Gamma\cap F$ consists of one point. Conversely, assume that $\Gamma\cap F$ consists of one point. Observe that no curve invariant by $h$ may be tangent to $\Gamma$ (since its image under the projection $S\to S/<h>$
locally splits  at the point of tangency). This implies that $F$ is singular and that $F\cap \Gamma$ is the singular point of $F$. If $F$ was not twisted by $h$, there would be three $h$-invariant tangent directions at this point, a contradiction. So $F$ is twisted by $h$ and we  get Assertion $1$.

Assume now that $F\cap \Gamma$ consists of two points exchanged by $g$. This implies that $n$ is even. If $F$ is smooth, it is isomorphic to $\p^1$, so $h=g^n$ acts identically on $F$, which is impossible since the fixed locus of $h$ is smooth and contains $\Gamma$. Thus, $F$ is singular, twisted by $g$ and not by $h$. The converse being obvious, we get Assertion~$2$.

Assertion $3$ follows from Riemann-Hurwitz formula and Assertion $1$.

Let us prove Assertion $4$. Since $(g,S,\pi)$ is minimal, any singular fibre of $\pi$ is twisted by a power of $g$. If this is not $h$, then it is $g$ (and all its odd powers). Moreover, this may occur only for two fibres, since $g$ acts on the basis with two fixed points (except when $n=1$ and $g=h$, in which case Assertion $4$ is trivially true).
\end{proof}
\begin{lemm}\label{Lemm:GeomCBk}
In Lemma~$\ref{Lem:GeomCB}$, no one of the three possibilities for $r\in\{0,1,2\}$ can be excluded.\end{lemm}
\begin{rema}
For another description of the number $r$, see \cite[Proposition 6.5]{bib:BlaLin}, especially Assertion $4$.\end{rema}
\begin{proof}
We provide distinct examples of pairs $(\Gamma,g|_{\Gamma})$, use Proposition~$\ref{Prop:ConjNsigma}$ to obtain the existence of an element $g\in \dJo$ associated to the pair, and use Assertion $4$ of  Lemma~\ref{Lem:GeomCB} to obtain $r$ by counting the pairs of points in the same fibre exchanged by $g|_{\Gamma}$. 

In each example, $\Gamma$ is the curve whose affine part has equation \[y^2=(x^n-1)(x^n+1)\] in $\C^2$, $\pi|_{\Gamma}$ corresponds to the projection $(x,y)\mapsto x$ and $g|_\Gamma$ is induced by a diagonal automorphism of $\C^2$. To see what happens on the two points at infinity, we use the birational map $\psi\colon(x,y)\dasharrow (x^{-1},\im yx^{-n})$ which sends birationally  $\Gamma$ on itself.

\begin{enumerate}
\item
$g|_\Gamma$ is induced by $\alpha\colon(x,y)\mapsto (e^{2\im\pi/n}x,y)$. On $x=0$, the two points of $\Gamma$ are fixed by $g|_\Gamma$. Conjugating $\alpha$ by $\psi$, we obtain $\alpha^{-1}$, so the two points at infinity are fixed by $g|_{\Gamma}$. This implies that $r=0$.
\item
$g|_\Gamma$ is induced by $\alpha\colon(x,y)\mapsto (e^{2\im\pi/n}x,-y)$. On $x=0$, the two points of $\Gamma$ are exchanged by $g|_\Gamma$. Conjugating $\alpha$ by $\psi$, we obtain $\alpha^{-1}$, so the two points at infinity are also exchanged by $g|_{\Gamma}$. This implies that $r=2$.
\item
$g|_\Gamma$ is induced by $\alpha\colon(x,y)\mapsto (e^{\im\pi/n}x,y)$. On $x=0$, the two points of $\Gamma$ are fixed by $g|_\Gamma$. Conjugating $\alpha$ by $\psi$, we obtain $(x,y)\mapsto (e^{-\im\pi/n}x,-y)$, so the two points at infinity are exchanged by $g|_{\Gamma}$. This implies that $r=1$.
\end{enumerate}
\end{proof}
Note that in \cite[Theorem 5.7]{bib:DoI}, Assertion (1) supposes that $r=0$. This is false, as the above examples show; the wrong argument is those given at the end of the proof of \cite[Lemma 5.6]{bib:DoI}, it is said "Since $g'$ is an even power it cannot switch any components of fibres". Although this argument is incorrect, the proof of \cite[Lemma 5.6]{bib:DoI} can be easily corrected by the fact that if $g'$ does not twist any singular fibre, $g'$ belongs to $G_0$ (in the notation of \cite{bib:DoI}), a case avoided by hypothesis. But the false argument is used in the proof of \cite[Theorem 5.7]{bib:DoI}. In fact, removing the sentence "and switches the components in all fibres" in this Theorem corrects it. This has been done in the new version of the paper that is available on the ArXiv.

\subsection{The proof of Theorem~\ref{Thm:ExplicitdJo}}
Proposition \ref{Prp:ConjIFFsameAction} gives the classification of conjugacy classes of non-diagonalisable elements of finite order of $\dJo$ in this group (note that for diagonalisable elements, the result is a simple exercise in linear algebra). The main invariant is the pair $(\Gamma,g|_{\Gamma})$ associated to $g$. If $\Gamma$ is an irreducible curve of positive genus, the pair is also an invariant of conjugacy in the Cremona group, but this is not the case otherwise (when the bisection $\Gamma$ is the union of $1$ or $2$ irreducible rational curves). In this case, the element $g$ is in fact conjugate to a linear diagonal automorphism of $\p^2$ (or $\p^1\times\p^1$) in the Cremona group, although it is not diagonalisable in the de Jonqui\`eres group. This was proved in  \cite{bib:BlaLin}, and was extended to finite Abelian groups. Using Propositions~\ref{Prop:ConjNsigma} and \ref{Prp:ConjIFFsameAction}, we could give an algebraic proof of this result, by considering the finitely many possibilities for $g$ and conjugating these directly by hand. However, we briefly recall here the main parts of the geometric proof.

\begin{prop}[\cite{bib:BlaLin}, Theorem 1]\label{Prp:Linearisation}
Let $g$ be an element of finite order $m\geq 1$ of $\Bir(\Pn)$ that leaves invariant a pencil of rational curves. 
The following conditions are equivalent:
\begin{enumerate}
\item[\upshape 1.]
No non-trivial power of $g$ fixes a curve of positive genus.
\item[\upshape 2.]
The element $g$ is birationally conjugate to an element of $\Aut(\Pn)$.
\item[\upshape 3.]
The element $g$ is birationally conjugate to the automorphism $(x:y:z)\mapsto (e^{2\im\pi/m}x:y:z)$ of $\Pn$.
\item[\upshape 4.]
The element $g$ is birationally conjugate to an element of $\Aut(\mathbb{F}_k)$, for some Hirzebruch surface $\mathbb{F}_k$.
\end{enumerate}

\end{prop}
\begin{proof}The automorphism $(x,y)\mapsto (e^{2\im\pi/m}x,y)$ of $\C^2$ extends to an automorphism of $\mathbb{F}_k$ for any $k$, so $(3)$ implies $(4)$.
The implication $(4)\Rightarrow (2)$ is an easy exercise, using elementary links of conic bundles at points fixed by $g$.  
The implication $(2)\Rightarrow (3)$ was proved in \cite{BeaBla} using the action of $\GL(2,\mathbb{Z})$ on $(\C^{*})^2$ via the action by conjugation of the group of monomial elements of $\Bir(\Pn)$ on the group of diagonal elements of $\Aut(\Pn)$, and was then generalised to any dimension and to elements of non-necessarily finite order in \cite{bib:BlaManMat}.

The hardest part consists in proving that $(1)$ implies the other conditions. Let us explain the argument.

If $g$ is diagonalisable in the de Jonqui\`eres group, then we are done. We may assume that $g$ has order $2n$ and acts on the basis of the fibration with order $n$. We view $g$ as an automorphism of a conic bundle $(S,\pi)$, and assume that $g$ acts minimally on the conic bundle, i.e.\ that every singular fibre is twisted by a power of $g$ (we say that an automorphism twists a singular fibre $F=F_{1}\cup F_{2}$ if it exchanges the components $F_{1}$ and $F_{2}$). If $\pi$ has no singular fibre, $S$ is a Hirzebruch surface and we obtain $(4)$. A quick computation (see \cite[Lemma~5.1]{bib:DoI}) shows that if $\pi$ has $1$, $2$ or $3$ singular fibres, then we may contract a $g$-invariant set of disjoint sections of self-intersection $-1$. This conjugates $g$ to an automorphism of a del Pezzo surface of degree $\geq 6$. A study of each possibility shows that $g$ is conjugate to an automorphism of $\p^2$ (\cite[Proposition 9.1]{bib:BlaLin}). 

We may thus assume that $\pi$ has at least $4$ singular fibres.
We have $\pi g=\bar{g}\pi$ for some automorphism $\bar{g}$ of $\p^1$ of order $n$. Since the involution $g^n=\sigma$ fixes no curve of positive genus, it can twist $0$ or $2$ singular fibres, depending on whether it fixes a reducible or an irreducible curve (\cite[Lemma~6.1]{bib:BlaLin}). The only possibility is that $\pi$ has $4$ singular fibres, two twisted by $\sigma$ and two by $g$ (which are the fibres of the points fixed by $\bar{g}$). In particular, $n=2$, and $g$ permutes the two fibres twisted by $\sigma$. Computing the intersections in the Picard group, we find that if $g^n=\sigma$ twists $2k$ fibres and $g$ twists $r$ fibres, then $2k/n\equiv r \pmod{2}$ (\cite[Proposition~6.5]{bib:BlaLin}). This creates a problem of parity, since $r=n=2$ and $k=1$.  
\end{proof}
We are now able to summarise the results of this section, by proving Theorem~\ref{Thm:ExplicitdJo}.
\begin{proof}[Proof of Theorem~$\ref{Thm:ExplicitdJo}$]
Assertion 1 follows from Proposition~\ref{Prp:Linearisation}; let us prove Assertion 2.
We fix some non-linearisable de Jonqui\`eres element, that we call $g$. After conjugation, we may assume that $g$ is an element of type (2) of Proposition~$\ref{Prop:JonqConj2cases}$, i.e.\ $g=\left(\rho, \left(\begin{array}{cc} 1 & 0 \\ 0& \xi\end{array}\right)\right)$, where $\xi \in \C^{*}$ has finite order $n\geq 1$, and $g^n=\sigma=\left(\begin{array}{cc} 0 & p(x^n) \\ 1& 0\end{array}\right)\in \PGL(2,\C(x))$. Since $\rho$ commutes with $\sigma$ (Proposition~$\ref{Prop:JonqConj2cases}$), $\rho$ is equal to $\left(\begin{array}{cc} a& (-1)^\delta bh \\ b& (-1)^\delta a\end{array}\right)\in \PGL(2,\C(x))$, for some $a,b\in \C(x)^{*}$, $\delta\in\{0,1\}$, by Lemma~\ref{Lem:Centraliser}.
The curve $\Gamma$ fixed by $\sigma$ has equation $y^2=p(x^n)$ and has positive genus by Proposition~$\ref{Prp:Linearisation}$. The action of $g$ on $\Gamma$ has order $n$, preserves the fibers of the $g_1^2$, and is not a root of the involution associated to the $g_1^2$. Furthermore, any such action can be obtained in this way, by Proposition~$\ref{Prop:ConjNsigma}$. The $1$-to-$1$ correspondence between conjugacy classes and actions is provided by Proposition~$\ref{Prp:ConjIFFsameAction}$.
\end{proof}
\section{Cyclic groups, not of de Jonqui\`eres type}\label{Sec:CyclicGrNotDJ}
In this section we study cyclic groups of finite order of $\Bir(\Pn)$ which do not preserve any pencil of rational curves. We remind the reader of the following classical result:
\begin{prop}
Let $g\in \Bir(\Pn)$ be an element of finite order, not of de Jonqui\`eres type. Then, there exists a birational map $\phi:\Pn\dasharrow S$ such that \begin{enumerate}
\item[\upshape 1.]
$S$ is a del Pezzo surface (projective smooth surface such that $-K_S$ is ample);
\item[\upshape 2.]
the conjugate $h=\phi g \phi^{-1}$ is an automorphism of $S$;
\item[\upshape 3.]
$\rkPic{S}^h=1$.\end{enumerate}
\end{prop}
\begin{proof}
Since $g$ has finite order, we may conjugate it to an automorphism $h$ of some projective smooth rational surface $S$ (see for example \cite{bib:DFE}). After contracting invariant sets of $(-1)$-curves, we may assume that the pair $(h,S)$ is minimal. Then, by Mori theory \cite{bib:Man} either $S$ is a del Pezzo surface and $\rkPic{S}^h=1$ or $h$ preserves a rational fibration; this latter possibility is excluded by hypothesis. 
\end{proof}
We may restrict our attention to the study of del Pezzo surfaces of degree $1$, $2$ or $3$, as the following simple lemma  shows:
\begin{lemm}\label{Lem:Deg4}
Let $S$ be a del Pezzo surface of degree $\geq 4$, and let $g\in \Aut(S)$. Then, $g$ preserves a pencil of rational curves.
\end{lemm}
\begin{proof}
By the Lefschetz formula, there exists at least one point of $S$ which is fixed by $g$  \cite[page 248, Corollary 1]{bib:AtiBott}.
If $S\cong \mathbb{P}^1\times\mathbb{P}^1$, the blow-up of a fixed point conjugates $g$ to an automorphism of a del Pezzo surface of degree $7$. 
We may thus assume the existence of a birational morphism $\eta:S\rightarrow \Pn$, which is the blow-up of $r$ points of $\Pn$, for $0\leq r\leq 5$. Then, $g$ preserves the linear system of curves equivalent to the anticanonical divisor $-K_S$, which are the strict pull-backs by $\eta$ of the cubics passing through the $r$ blown-up points. Furthermore, $g$ also preserves the linear subsystem $\Lambda_p$ corresponding to the cubics being singular at $\eta(p)$, where $p\in S$ is fixed by $g$. Since the dimension of $\Lambda_p$ is $6-r\geq 1$, and the action of $g$ on it is linear, we can find a subsystem of dimension $1$, invariant by $g$; this completes the proof.
\end{proof}
\begin{rema}
This simple result  is also a consequence of the huge work of {\upshape \cite{bib:DoI}}.
\end{rema}

Recall (\cite{bib:Kol}, Theorem III.3.5) that a del Pezzo surface of degree $3$ (respectively $2$, $1$) is isomorphic to a smooth hypersurface of degree $3$ (respectively $4$, $6$) in the projective space $\mathbb{P}^3$ (respectively in $\mathbb{P}(1,1,1,2)$, $\mathbb{P}(1,1,2,3)$). Furthermore, in each of the $3$ cases, any automorphism of the surface is the restriction of an automorphism of the ambient space. 

For any automorphism $g$ of a del Pezzo surface of degree $1$, $2$ or $3$, we can choose coordinates $w,x,y,z$  on the projective spaces such that $g$ is a diagonal automorphism $(w:x:y:z)\mapsto (\alpha w:\beta x:\gamma y:\delta z)$ -- that we will denote by $\DiaG{\alpha}{\beta}{\gamma}{\delta}$ -- of the weighted projective space which preserves the equation of the  surface, this equation being one of the following:
\begin{eqnarray}
0&=&L_3(w,x,y,z),\\
w^2&=&L_4(x,y,z),\\
w^2&=&z^3+z\cdot L_4(x,y)+L_6(x,y),
\end{eqnarray}
where $L_i$ denotes a homogeneous form of degree $i$.

 A systematic study of all automorphisms is therefore possible, and was carried out in \cite{bib:Seg} (for degree $3$), \cite{bib:DolgTopic} and \cite{bib:Bar} (degree $2$),  and especially in \cite{bib:DoI}  where all the results are summarised (this was also done in my PhD thesis, see \cite{bib:JBCR}). 
 
Now, to eliminate the groups which are not of de Jonqui\`eres type is more subtle. This was done in \cite{bib:DoI} (and in my thesis), using tools like the Lefschetz formula or Weyl groups.
Summing up, we get that any finite cyclic subgroup of $\Bir(\Pn)$ which is not of de Jonqui\`eres type is conjugate to a group of Table~$\ref{ListN1}$. Note that the $29$ families of Table~$\ref{ListN1}$ are those of \cite{bib:JBCR}, and also those of \cite[Table 9]{bib:DoI}, except for one family, present in \cite[Table 9]{bib:DoI} under the notation $D_5$, which consists of one group of order $8$ acting on a del Pezzo surface of degree $4$, and is in fact of de Jonqui\`eres type (by Lemma~\ref{Lem:Deg4}).

To determine the conjugacy classes among the $29$ families (which is done in the proof of Theorem~\ref{Thm:ConjAmong29} below), we need a few classical results on plane cubic curves, of which we remind the reader:
\begin{lemm}\label{Lem:CubicHesse}
Let $\Gamma\subset \Pn$ be a smooth cubic curve, and let $\Aut(\Pn)_{\Gamma}$ be the group of automorphisms of $\Pn$ which preserve $\Gamma$. Then, the following occur.
\begin{enumerate}
\item[\upshape 1.]
There exists $g\in \Aut(\Pn)$ such that $g(\Gamma)$ has Hesse form, i.e.\ that its equation is $x^3+y^3+z^3 +\lambda xyz=0$, for some $\lambda\in \C$.
\item[\upshape 2.]
If $\Gamma'\subset \Pn$ is isomorphic to $\Gamma$, then there exists $g\in \Aut(\Pn)$ such that $g(\Gamma)=\Gamma'$.
\item[\upshape 3.]
The group $\Aut(\Pn)_{\Gamma}$ acts transitively on the $9$ inflexion points of $\Gamma$.
\item[\upshape 4.]
If $\tau,\sigma\in \Aut(\Pn)_{\Gamma}$ induce two translations on  $\Gamma$ which are conjugate in $\Aut(\Gamma)$, then $\tau$ and $\sigma$ are conjugate in $\Aut(\Pn)_{\Gamma}$.
\end{enumerate}
\end{lemm}
\begin{proof}
The first two assertions are very classical, proofs can be found in \cite{bib:ADol}. Putting $\Gamma$ in its Hesse form, the $9$ inflexion points become easy to compute; these are the orbits by $[0:1:-1]$ of the subgroup of $\Aut(\Pn)$ generated by the permutation of coordinates and $(x:y:z)\mapsto (x:\omega y:\omega^2 z)$, where $\omega$ is a primitive $3$-rd root of unity. This yields Assertion $3$. The group $\Aut(\Gamma)$ is isomorphic to $\Gamma\rtimes \Z{m}$, for $m\in \{2,4,6\}$, and is generated by the translations (the abelian group $\Gamma$) and one element of order $m$ with fixed points. To prove Assertion $4$, it thus suffices to find in $\Aut(\Pn)_{\Gamma}$ an element of order $m$ which acts on $\Gamma$ with fixed points. We put $\Gamma$ in its Hesse form, and see that $(x:y:z)\mapsto (y:x:z)$ is suitable if $m=2$. If $m=6$, then we may assume that $\Gamma$ is the Fermat cubic $x^3+y^3+z^3=0$, and $(x:y:z)\mapsto (y:x:\omega z)$ is suitable. If $m=4$, we see that $(x:y:z)\mapsto (x:-y:\im z)$ acts with fixed points on the smooth cubic curve $\Gamma'$ of equation $xz^2+y^3+x^2y$; we use Assertion $2$ to conclude.
\end{proof}
\begin{proof}[Proof of Theorem~\ref{Thm:ConjAmong29}]
The first assertion follows from the beginning of this section.

Recall that the set of points fixed by an automorphism of a surface is smooth (as can be shown by local analytic linearisation). We can therefore apply the adjunction formula, and the canonical embeddings into weighted projective spaces to deduce the genus of the fixed curves.

(i) On a del Pezzo surface of degree $3$, i.e.\ a cubic surface in $\mathbb{P}^3$, any fixed curve is elliptic, and any birational morphism to $\mathbb{P}^2$ sends it on a smooth cubic curve (for further details and generalisations, see \cite{bib:JBMich}). 

(ii) A del Pezzo surface of degree $2$ is canonically embedded into $\mathbb{P}(2,1,1,1)$ as a surface of equation $w^2=L_4(x,y,z)$. The curve given by $w=0$ is equivalent to $-2K_S$, and is a Geiser curve isomorphic to a smooth quartic plane curve (but is sent by any birational morphism to $\Pn$ on a sextic with $7$ double points, see \cite{bib:BPV2} for further details). The other fixed curves are elliptic curves, equivalent to $-K_S$.

(iii) A del Pezzo surface of degree $1$ is canonically embedded into $\mathbb{P}(3,1,1,2)$ as a surface of equation $w^2=z^3+zL_4(x,y)+L_6(x,y)$. The curve given by $w=0$ is equivalent to $-3K_S$, and is a Bertini curve. The curve given by $z=0$ is equivalent to $-2K_S$ and is a (hyperelliptic) curve of genus $2$. The other possible fixed curves are elliptic curves, equivalent to $-K_S$.

The observations (i), (ii), (iii) and a simple examination of each of the $29$ cases directly yields Assertion $4$. 

Let us prove Assertion~$2$. Let $g$ be a generator of a group $G$ of one of the $29$ families; we distinguish three cases. (a) If $G$ belongs to family $\#9$, then $g^3$ fixes an elliptic curve, and $g$ acts on this curve as a translation of order $3$, so does not preserve the fibers of any $g_1^2$. (b) If $G$ belongs to family $\#n$ for $n\in \{3,4,5,7,8,12,14,17,19,21,22,25\}$, the $m$-torsion of $G$ fixes a curve of genus $1$ or $2$, for some integer $m\geq 3$. (c) In the remaining families, some non-trivial element of $G$ fixes an non-hyperelliptic curve (Geiser or Bertini curve). In cases (a) through (c), and for any de Jonqui\`eres element $h$, we have $\NFCA(g)\not=\NFCA(h)$  (by the description of $\NFCA(h)$ given in Theorem~\ref{Thm:ExplicitdJo}); this proves Assertion 2.

Since the subgroups of two different families of the same order fix different types of curves, Assertion~$3$ is clear. Note that this could also be proved using the fact that each family appears in \cite[Table 9]{bib:DoI} with the conjugacy class of its action in the Weyl group, and that all of them are different.

There remains to study each of the $29$ families and to show that if two groups belonging to the same family have similar $\NFCA$ invariants, there exists an isomorphism between the two surfaces on which the groups act, that conjugates the two groups; the explicit parametrisations will follow from this study.

Let us explain the approach and notation of the remaining part. For each family, we take a group $G$, acting on the surface $S$, embedded in a weighted projective space ($\mathbb{P}^3$, $\mathbb{P}(2,1,1,1)$ or $\mathbb{P}(3,1,1,2)$), such that the pair $(G,S)$ represents an element of the family. We choose an integer $r$, denote by $\Gamma\subset S$ the curve fixed by the $r$-torsion of $G$, and explain which kind of curve it is. Then, we take another group $G'$ which acts on $S'$, whose $r$-torsion fixes $\Gamma'\subset S'$ (we only add a prime to the group, surface and parameters), and assume that there exists an isomorphism $\alpha:\Gamma\rightarrow \Gamma'$ which conjugates the group $H=G|_{\Gamma}$ to $H'=G'|_{\Gamma'}$. We prove that $\alpha$ can be chosen so that it extends to an automorphism of the weighted space which sends $S$ on $S'$ and conjugates $G$ to $G'$. For some families (when $G|_{\Gamma}$ is trivial, or unique as in $\#8$, $\#12$ and $\#13$) the same results work by assuming only that $\Gamma$ is isomorphic to $\Gamma'$, without taking into account $H,H'$; this simplifies the parametrisation.

In families $\# 1$ through $\# 5$, and in $\#7$ and $\#17$, each non-trivial element of the group fixes the same curve, and the isomorphism class of the curve determines the isomorphism class of the surface, so $r=|G|$. For example, in family $\#5$, suppose that $S,S'$ are surfaces given by equations $w^2=L_4(x,y)+z^4$ and $w^2=L_4'(x,y)+z^4$ respectively, and that $G,G'\subset \Aut(S)$ are generated by $z\mapsto \im z$. The two corresponding curves are elliptic curves, double coverings of $\mathbb{P}^1$ by means of the projection on $(x:y)$, ramified over the four roots of respectively $L_4$ and $L_4'$. If the curves are isomorphic, there exists $\beta\in\Aut(\mathbb{P}^1)$ such that $\beta(L_4)=L_4'$. Clearly $\beta$ extends to $\mathbb{P}(2,1,1,1)$, and conjugates $G$ to $G'$. A similar simple argument works for families $\#1$ through $\#4$, and for $\#7$ and $\#17$.

[$\#6$, $r=2$] The $2$-torsion fixes the Bertini curve $\Gamma$, and $H$ is generated by an involution with fixed points. Since the curve determines the surface, we may assume that $G$, $G'$ act on the same surface $S$. The action of $\Aut(S)$ on $\Gamma$  yields the classical exact sequence $1\rightarrow \Z{2} \rightarrow \Aut(S) \rightarrow \Aut(\Gamma) \rightarrow 1$; consequently if $H,H'\subset\Aut(\Gamma)$ are conjugate, so are $G,G'\subset \Aut(S)$.

[$\#8$, $r=3$] The curve $\Gamma$ is the elliptic curve whose equation in the plane $w=0$ of $\mathbb{P}^3$ is $F=x^3+y^3+xz^2+\lambda yz^2$, and  $H\subset \Aut(\Gamma)$ is generated by an involution which fixes four points, one of these being the inflexion point $(0:0:1)$. Since $\Gamma\cong \Gamma'$, there exists $\beta\in \Aut(\Pn)$ such that $\beta(F)=F'$ and we may assume that $\beta$ fixes the inflexion point $(0:0:1)$ (Lemma~\ref{Lem:CubicHesse}, Assertions 2 and 3). The extension of $\beta$ to an automorphism of $\mathbb{P}^3$ therefore conjugates $G$ to $G'$. 

[$\#9$, $r=2$] The curve $\Gamma$ is elliptic and its equation in the plane $x=0$ of $\mathbb{P}^3$ is $w^3+y^3+z^3+\lambda yz^2$; the group $H\subset \Aut(\Gamma)$ is generated by a translation of order $3$, induced by the automorphism $(w:y:z)\mapsto (w:\omega y:\omega^2z)$ of the plane. Since $\Gamma\cong \Gamma'$, there exists $\beta\in \Aut(\Pn)$ such that $\beta(\Gamma)=\Gamma'$ (Lemma~\ref{Lem:CubicHesse}, Assertion 2).  Then, $\beta$ conjugates $H$ to a subgroup of $\Aut(\Gamma')$ which is conjugate to $H'$, and up to a change of $\beta$ we may therefore assume that $\beta$ conjugates $H$ to $H'$(Lemma~\ref{Lem:CubicHesse}, Assertion 4). Since $\beta$ commutes with $\rho:(w:y:z)\mapsto (w:\omega y:\omega^2z)$, we may choose -- after composition by a power $(w:y:z)\mapsto (y:w:z)$ -- that $\beta$ is diagonal,  which implies that it is a power of $\rho$. Consequently, $G$ and $G'$ are  equal.

[$\#10$, $\#11$, $r=2$] In both families, $\Gamma$ is the Geiser curve of the surface, $H\subset \Aut(\Gamma)$ has order $3$ and fixes respectively $4$ or $2$ fixed points of $\Gamma$. The isomorphism class of the Geiser curve determines that of the surface, and the classical exact sequence $1\rightarrow \Z{2}\rightarrow \Aut(S)\rightarrow \Aut(\Gamma)\rightarrow 1$ shows that the conjugacy class of $H\subset \Aut(\Gamma)$ determines the conjugacy class of $G$ in $\Aut(S)$.

[$\#12$, $r=3$] The curve $\Gamma$ has equation $w^2=y^4+z^4+\lambda y^2z^2=F_4(y,z)$ in the weighted plane $x=0$, and is therefore elliptic; furthermore $H\subset\Aut(\Gamma)$ is generated by a translation of order $2$. Denote by $\nu\in\Aut(\Gamma)$ (respectively $\nu'\in\Aut(\Gamma')$) the involution induced by $w\mapsto -w$; this is the involution which corresponds to the $g_1^2:\Gamma\rightarrow \mathbb{P}^1$ given by $(w:y:z)\dasharrow (y:z)$. By hypothesis, there exists an isomorphism $\alpha:\Gamma\rightarrow \Gamma'$ that conjugates $H$ to $H'$. Since all the non-translation involutions are conjugate in $\Aut(\Gamma)$, there exists an isomorphism $\beta:\Gamma\rightarrow \Gamma'$ that conjugates  $\nu$ to $\nu'$. Let us now prove that we can choose $\beta$ so that it also conjugates $H$ to $H'$. If $\Gamma$ is a general elliptic curve, this is clear since $H$ is in the center of $\Aut(\Gamma)$ and $H$, $H'$ are conjugate by hypothesis. Otherwise, we can choose a root of $\nu$ (of order $4$ or $6$) which conjugates $\beta H\beta^{-1}$ to $H'$, which completes the argument. Since $H$ conjugates $\nu$ to $\nu'$, it is an isomorphism between the fibres of the corresponding two $g_1^2$, and therefore extends to an automorphism of the plane $x=0$ (isomorphic to $\mathbb{P}(2,1,1)$) that sends $w^2-F_4(y,z)$ onto $w^2-F_4'(y,z)$. The extension of this automorphism to an automorphism of $\mathbb{P}(2,1,1,1)$ follows directly.

[$\#13$, $\#14$, $r=3$] In both families, $\Gamma$ has genus $2$ and belongs to the weighted plane $z=0$. Furthermore, $H$ is generated by an involution which is respectively that induced by the $g_1^2$, with $6$ fixed points, and an involution with $4$ fixed points. Since $\alpha:\Gamma\rightarrow \Gamma'$ preserves the $g_1^2$, it extends to the plane $z=0$ and hence to $\mathbb{P}(3,1,1,2)$.
Note that in fact, for family $\#13$, $r=2$ is also possible.

[$\#15$, $\#16$, $\#18$, $\#20$, $\#23$, $r=2$] In these five families, $\Gamma$ is the Bertini curve, and $H$ is a cyclic group of order respectively $3$, $3$, $4$, $5$,  $6$, which fixes respectively $3$, $1$, $2$, $4$, $2$ points of $\Gamma$. Once again, the isomorphism class of $\Gamma$ determines that of the surface and the classical exact sequence induced by the action on the Bertini curve yields the result.

Each of the remaining families ($\#19$, $\#21$, $\#22$, and $\#24$, through  $\#29$) contains a single element.
\end{proof}
\section{Elements of finite order, not of de Jonqui\`eres type}\label{Sec:EltGrNotDJ}
In this section, we distinguish the generators of the cyclic groups which are not of de Jonqui\`eres type.

\begin{proof}[Proof of Theorem~\ref{Thm:EltGrNotDJ}]
Suppose that $g,h$ are conjugate by some birational transformation $\varphi$ of~$S$. Then, since $\varphi$ is $G$-equivariant, we may factorise it into a composition of automorphisms and elementary $G$-equivariant links (\cite{bib:IskMori}, Theorem 2.5). Since our surface is of Del Pezzo type ($S\in \{\mathbb{D}\}$ in the notation of \cite{bib:IskMori}), the first link is of type $\mathrm{I}$ or $\mathrm{II}$. The classification of elementary links (\cite{bib:IskMori}, Theorem 2.6) shows that the only possiblity for the link is to be the Geiser or Bertini involution of a surface obtained by the blow-up of one or two points invariant by $G$. A Geiser (respectively Bertini) involution of a Del Pezzo surface of degree $2$ (respectively $1$) commutes with any automorphism of the surface, thus the elementary link conjugates $g$ to itself. Consequently, $g$ and $h$ are conjugate by some element of $\Aut(S)$.

For $n\in\{1,2\}$ there is only one generator since $G\cong \Z{2}$. 

If $n=9$ (respectively $n=11$), then the automorphism $(w:x:y:z)\mapsto (w:x:z:y)$ of $\mathbb{P}^3$ (respectively of $\mathbb{P}(2,1,1,1)$) induces an element of $\Aut(S)$ which conjugates the two generators of $G\cong\Z{6}$.

If $n\not=\{6,9,11,16\}$, two distinct generators of $G$ are induced by two automorphisms of the weighted space which are diagonal, with different eigenvalues (up to a weighted multiplication). Consequently, distinct generators of $G$ are not conjugate by an element of $\Aut(S)$.

If $n=16$, we write explicitly the form $L_2$ of degree~$2$, and obtain the expression $w^2=z^3+\lambda x^2y^2z+(ax^6+bx^3y^3+cy^6)$ for the equation of $S$ in $\mathbb{P}(3,1,1,2)$, with some $a,b,c,\lambda \in \C$. Furthermore, $ac\not=0$ since $S$ is not singular. Choose $\alpha,\gamma,\mu\in \C^{*}$ such that $\alpha^3=a,\gamma^3=c,\mu^2=ac$. Then, the automorphism $(w:x:y:z)\mapsto (\mu w:\gamma y:\alpha x:\alpha\beta y)$ of $\mathbb{P}(3,1,1,2)$ induces an automorphism of $S$ which conjugates the two generators of $G\cong\Z{6}$.

The remaining case is $n=6$. We write explicitly the two forms $L_2,L_2'$ of degree~$2$, and obtain  $w^2=z^3+z(ax^4+bx^2y^2+cx^4)+xy(a'x^4+b'x^2y^2+c'y^4)$ for the equation of $S$ in $\mathbb{P}(3,1,1,2)$, with some $a,b,c,a',b',c'\in \C$. Since $S$ is not singular, $a,a'$ are not both zero, and $b,b'$ are not both zero.

Assume now that the two generators $g_1,g_2$ of $G$ are conjugate by $\tau\in\Aut(S)$. 
Recall that $g_1$, $g_2$ are induced respectively by the automorphisms $\overline{g_1}:(w:x:y:z)\mapsto (\im w:x:-y:-z)$ and $\overline{g_2}:(w:x:y:z)\mapsto (\im w:x:-y:-z)$ of $\mathbb{P}(3,1,1,2)$. 
Furthermore, $\tau$ extends to $\overline{\tau} \in \Aut(\mathbb{P}(3,1,1,2))$, which conjugates $\overline{g_1}$ to $\overline{g_2}$, and thus $\overline{\tau}$ is equal to $(w:x:y:z)\mapsto (\mu w:\gamma y:\alpha x:\nu z)$, for some $\alpha,\gamma,\mu,\nu\in \C^{*}$. 
 Consequently, the expression $\mu^2 w^2 =\nu^3 z^3+\nu z(\alpha^4c x^4+\alpha^2\gamma^2bx^2y^2 +\gamma^4a y^4)+\alpha\gamma xy(\alpha^4 c' x^4+\alpha^2\gamma^2b' x^2y^2+\gamma^4a'y^4)$ is a multiple of the equation of $S$. 
 
 This shows that the two generators of $G$ are conjugate in $\Aut(S)$ if and only if there exist $\alpha,\gamma,\mu,\nu\in \C^{*}$ such that the following two vectors of $\C^8$ are linearly dependent:
 \[\begin{array}{rrrrrrrrrl}(\!\!&1,& 1,& a,& b,& c,& a',&b',&c'&\!\!),\vspace{0.1 cm}\\
 (\!\!&\!\mu^2,& \!\nu^3,& \!\nu\cdot \alpha^4c,& \!\nu\cdot \alpha^2 \gamma^2b,& \!\nu\cdot \gamma^4a,&\!\alpha\gamma\cdot \alpha^4c',&\!\alpha\gamma\cdot \alpha^2\gamma^2 b',&\!\alpha\gamma\cdot \gamma^4a'&\!\!).\end{array}\]
 
 We now consider different cases, and decide whether the two generators are conjugate. 
 
 a) Assume that $a=c=0$, which implies that $a',c'\in \C^{*}$. Then, choosing $\alpha,\gamma,\mu,\nu$ such that $\alpha^2=a',\gamma^2=c',\nu=\alpha\gamma,\mu^2=\nu^3$ gives a positive answer.
 
 b) The case $a'=c'=0$ is similar. We choose $\alpha,\gamma,\mu,\nu$ such that $\alpha^2=a,\gamma^2=c,\nu=\alpha\gamma,\mu^2=\nu^3$.
 
 c) If $a=a'=0$, then $cc'\not=0$, and the answer is negative. The same occurs for $c=c'=0$, which implies that $aa'\not=0$.
 
 d) There remain the cases where $a,c,a',c'\in \C^{*}$. Linear dependence implies the following conditions:

 \begin{itemize}
 \item[\upshape 1.] $(\alpha/\gamma)^4=(a/c)^2=(a'/c')^2 \Rightarrow a/c=\pm a'/c'$;
 \item[\upshape 2.]
 $\nu=c/a\cdot c'/a'\cdot (\alpha/\gamma)^4 \cdot \alpha\gamma=(c/a)\cdot (a'/c')\cdot \alpha\gamma$;
 \item[\upshape 3.]
 $ab\gamma^2=bc \alpha^2 \Rightarrow b\left((\alpha/\gamma)^2-a/c\right)=0$;
 \item[\upshape 4.]
  $a'b'\gamma^2=b'c'\alpha^2 \Rightarrow b'\left((\alpha/\gamma)^2-a'/c'\right)=0$.
 \end{itemize}
 
 If $a/c=a'/c'$, choosing $\alpha,\gamma,\mu,\nu$ such that $\alpha^2=a,\gamma^2=c,\nu=\alpha\gamma,\mu^2=\nu^3$ gives a positive answer. If $a/c=-a'/c'$ and $bb'\not=0$, the conditions above show that the answer is negative. If $a/c=-a'/c'$ and $b'=0$ (respectively $b=0$), choosing  $\alpha,\gamma,\mu,\nu$ such that $\alpha^2=a,\gamma^2=c$ (respectively $\alpha^2=a'$, $\gamma^2=c'$) and $\nu=-\alpha\gamma,\mu^2=\nu^3$ gives a positive answer.
\end{proof}
\section{The importance of the $\NFCA$ invariant}\label{Sec:ProofNFCA}
\begin{proof}[Proof of Theorem~$\ref{Thm:NFCA}$]
Let $g,h\in \Bir(\Pn)$ be two elements of the same finite order, such that $\NFCA(g)=\NFCA(h)$. 

Assume that $g,h$ are both of de Jonqui\`eres type. If neither $g$ nor $h$ is linearisable, Assertion~$2$ of Theorem~$\ref{Thm:ExplicitdJo}$ shows that $g$ and $h$ are conjugate. If one of the two elements -- say g -- is linearisable, then $\NFCA(g)$ is a sequence of empty sets, and so is $\NFCA(h)$; Theorem~$\ref{Thm:ExplicitdJo}$ shows that $h$ is linearisable, and conjugate to $g$.

Assume now that at least one of the two elements is not of de Jonqui\`eres type. Assertion~2 of Theorem~\ref{Thm:ConjAmong29} implies that neither $g$ nor $h$ is of de Jonqui\`eres type. Then, Assertion~6 of Theorem~\ref{Thm:ConjAmong29} shows that the groups generated respectively by $g$ and $h$ are conjugate. However, in general $g$ and $h$ are not conjugate (Theorem~\ref{Thm:EltGrNotDJ}).
\end{proof}

\end{document}